\documentclass[11pt]{article}
\usepackage[english]{babel} 

\usepackage{amsmath, amsfonts, amssymb, amstext, amsthm} 
 \usepackage{stmaryrd}
 \usepackage{graphicx, subfigure} 
 \usepackage{color} 
 \usepackage{enumitem}

\def\R{\mathbb{R}}
\def\N{\mathbb{N}}
\def\e{\varepsilon}
\def\la{\lambda}
\def\vp{\varphi}

\numberwithin{equation}{section}

\newtheorem{Theorem}{Theorem}[section]
\newtheorem{Lemma}[Theorem]{Lemma}

\newtheorem{theorem}[Theorem]{Theorem}
\newtheorem{lemma}[Theorem]{Lemma}
\newtheorem{proposition}[Theorem]{Proposition}
\newtheorem{corollary}[Theorem]{Corollary}
\newtheorem{remark}[Theorem]{Remark}

\newtheorem{prop}[Theorem]{Proposition}

\DeclareMathOperator{\dist}{dist}

\newcommand{\cI}{{\mathcal I}}

\newcommand{\cO}{{\mathcal O}}
\newcommand{\cP}{{\mathcal P}}

 \newcommand{\cV}{{\mathcal V}}

\newcommand{\al}{\alpha}
\newcommand{\be}{\beta}
\newcommand{\ga}{\gamma}

\newcommand{\Ga}{\Gamma}
\newcommand{\de}{\delta}
\newcommand{\ep}{\epsilon}

\newcommand{\om}{\omega}

\newcommand{\Om}{\Omega}
\newcommand{\La}{\Lambda}

\newcommand{\codt}{\cdot} 

\newcommand{\ol}{\overline}
\newcommand{\ul}{\underline}

\newcommand{\upto}{\nearrow}
\newcommand{\downto}{\searrow}

\newcommand{\Lao}{\Lambda_{out}}

\newcommand{\hp}{\hat p}
\newcommand{\hq}{\hat q}

\title{Large-time behavior of solutions of parabolic 
  equations on the real line with convergent initial data III:
 unstable limit at infinity
} 
\author{Antoine Pauthier\\
 {\small Department of Mathematics, University of Bremen}\\
 {\small Postfach 22 04 40, 20359 Bremen}\\
 {}\\
{\small and}\\  
 \\
  Peter Pol\'{a}\v{c}ik\footnote{Supported in part by the NSF
    Grant DMS-1856491}  \\
{\small School of Mathematics, University of Minnesota}\\
{\small Minneapolis, MN 55455}
}
 
\date{}
\begin{document}
\maketitle

\vspace{-.5cm}
 \begin{center}
    {\small \em Dedicated to Eiji Yanagida \\
on the occasion of his 65th birthday} 
 \end{center}
 \medskip

\begin{abstract}
  This is a continuation, and conclusion, of our study
  of bounded solutions $u$ of the semilinear
  parabolic equation $u_t=u_{xx}+f(u)$ on the real line whose
  initial data  $u_0=u(\cdot,0)$ have finite limits
  $\theta^\pm$ as $x\to\pm\infty$. We assume that 
  $f$ is  a locally Lipschitz function on $\R$ satisfying 
  minor nondegeneracy conditions. Our goal is to describe the
  asymptotic behavior of $u(x,t)$ as $t\to\infty$.  In the first two
  parts of this series we mainly considered the cases where either
  $\theta^-\ne \theta^+$;
  or $\theta^\pm=\theta_0$ and $f(\theta_0)\ne0$; or else 
  $\theta^\pm=\theta_0$, $f(\theta_0)=0$, and $\theta_0$ is a stable
  equilibrium of the 
  equation $\dot \xi=f(\xi)$. In all these cases we proved that  the
  corresponding solution $u$ is
  quasiconvergent---if bounded---which is to say that all limit
  profiles of $u(\cdot,t)$  as  $t\to\infty$ are steady states.
  The limit profiles, or accumulation points, are taken in
  $L^\infty_{loc}(\R)$. In the present paper, we take on the
  case that $\theta^\pm=\theta_0$,
  $f(\theta_0)=0$, and $\theta_0$ is an 
  unstable equilibrium of the equation $\dot \xi=f(\xi)$.
  Our earlier  quasiconvergence theorem in this case involved some
  restrictive technical conditions on the solution, which
  we now remove. Our sole
  condition on $u(\cdot,t)$ is that it is nonoscillatory (has only
  finitely many  critical points) at some $t\ge 0$.
  Since it is known that oscillatory
  bounded solutions are not always
  quasiconvergent, our result is nearly optimal. 
\end{abstract}

{\emph{Key words}: Parabolic equations on $\R$, quasiconvergence,
  entire solutions, chains,  spatial
  trajectories,
  zero number}

\footnotesize
\tableofcontents
\normalsize
\section{Introduction and main results}

Consider the Cauchy problem
\begin{align}
  u_t=u_{xx}+f(u), & \qquad x\in\R,\ t>0, \label{eq:1}\\
  u(x,0)=u_0(x), & \qquad x\in\R, \label{ic:1}
\end{align}
where $f$ is a locally Lipschitz function on $\R$ and
$u_0\in C_b(\R):=C(\R)\cap L^\infty(\R).$ We denote by
$u(x,t,u_0),$ or simply $u(x,t)$ if there is no danger of
confusion, the unique classical solution of (\ref{eq:1}),
(\ref{ic:1})---to ensure the uniqueness, we require classical solutions to
satisfy  $u(\cdot,t)\in L^\infty(\R)$ as long as they are
defined---and
by $T(u_0)\in(0,+\infty]$ its maximal existence
time. If $u$ is
bounded on $\R\times[0,T(u_0))$, then necessarily $T(u_0)=+\infty,$
that is, the solution is global.

As in the previous two parts of this
paper series, \cite{p-Pauthier1,p-Pauthier2},
we examine solutions with
initial data $u_0$ taken in the space 
\begin{equation}\label{spacelimit}
  \mathcal{V}:=\left\{ v\in C_b(\R):\textrm{ the limits }
    v(-\infty),\,v(+\infty)\in \R \textrm{ exist}\right\}.
\end{equation}
In other words, we assume that
the initial datum $u_0$ has limits as $x\to\pm\infty.$
It is well known 
that the space $\mathcal V$ is invariant for
equation (\ref{eq:1}): if $u_0$ admits finite
limits as $x\to\pm\infty$, then so does $u(\cdot,t,u_0)$ for any $t\in
(0,T(u_0))$ (see Lemma \ref{le:limits} below; note that
the limits may vary with $t$).  So  it is natural to
consider $\mathcal V$ as a state space for equation \eqref{eq:1}.
Our goal is to understand the large-time behavior of
bounded solutions $u(\cdot,t)\in \cV$ and, in particular, 
to clarify if, in any fixed bounded interval,
the shape of $u(\cdot,t)$ at large times
is determined by steady states
of \eqref{eq:1}. To express this formally,   we
introduce the $\omega$-limit set of a bounded solution $u$:
\begin{equation}\label{defomega}
  \omega(u):=\left\{ \vp\in L^\infty(\R),\ u(\cdot,t_n)\to\vp \textrm{
      for some sequence }t_n\to\infty\right\}. 
\end{equation}
Here the convergence is in the topology of $L^\infty_{loc}(\R)$, that
is, the locally uniform convergence. By standard parabolic estimates,
the trajectory $\{ u(\cdot,t),\ t\geq1\}$ of a bounded solution is
relatively compact in $L^\infty_{loc}(\R).$ This implies that
$\omega(u)$ is nonempty, compact, and connected in $L^\infty_{loc}(\R)$,
and it attracts the solution in (the metric space)
$L^\infty_{loc}(\R)$:
$$
\textrm{dist}_{L^\infty_{loc}(\R)}\left(
  u(\cdot,t),\omega(u)\right)\underset{t\to\infty}{\longrightarrow}0.
$$
If the $\omega-$limit set reduces to a single element $\varphi$, then
$u$ is \textit{convergent:} $u(\cdot,t)\to\vp$ in $L^\infty_{loc}(\R)$
as $t\to\infty$.  Necessarily, $\vp$ is a steady state of \eqref{eq:1}.
If all functions $\vp\in\omega(u)$ are steady states of (\ref{eq:1}),
the solution $u$ is said to be \textit{quasiconvergent}.
Convergence and quasiconvergence both express a relatively tame
character of the solution in question. In both cases,
$u_t(\cdot,t)$ approaches zero locally uniformly on $\R$
as $t\to \infty$. For this reason, it is difficult to
numerically distinguish convergence from quasiconvergence
(analytically, convergence is characterized by the existence of
the improper Riemann integral of $u_t(x,t)$ on $[1,\infty)$
for each $x$).

For analogs of \eqref{eq:1} on bounded intervals under
Dirichlet, Neumann, Robin, or periodic 
boundary conditions, or 
sometimes even for \eqref{eq:1} itself when solutions
in suitable energy
spaces are considered,  quasiconvergence of solutions can be
established by means of energy estimates (see, for example,
\cite{Feireisl_NoDEA97}). However, the existence of
the limits $u_0(\pm\infty)$ alone is not sufficient for
quasiconvergence. As shown in  \cite{P:examples,P:unbal},
bounded solutions in $\mathcal V$ which are not quasiconvergent do
exist. (We emphasize here that  the locally uniform convergence
is taken in the definition of the $\om$-limit set and
the corresponding notion of
quasiconvergence; if the uniform convergence is taken instead, the
existence of bounded solutions which are not quasiconvergent
is rather trivial).
Moreover, the existence of such solutions is not an exceptional
phenomenon at all; it is guaranteed by a robust condition on  $f$, namely
the existence of a bistable interval.
Note, however, that steady states
are not completely irrelevant for the behavior of 
non-quasiconvergent solutions. A result of \cite{Gallay-S,Gallay-S2}
shows that the $\om$-limit set of any bounded solution of \eqref{eq:1} 
contains at least one steady state. There are convergence and
quasiconvergence results for various specific classes
of solutions of \eqref{eq:1}, see
\cite{Du_Matano,p-Du,p-Fe,Li-L-L,Matano_Polacik_CPDE16,p-Ma:1d-p2,
  Polacik_terrasse, Risler_1,Risler:terrace-1d};
 overviews can be found
in \cite{P:quasiconv-overview, p-Pauthier2}).  More recent results
include a convergence theorem of  \cite{Ding-M:compact-sp}
for positive solutions of  periodic versions of \eqref{eq:1} with compact
initial data and a description of the large-time behavior of
entire solutions with localized past \cite{Hamel-Ni}  (the latter 
paper deals with equations on $\R^N$ and
its introduction also contains an overview of earlier results
for multidimensional parabolic problems). 

Our study of solutions in $\mathcal V$, which we conclude in
this paper, yields a rather complete information on the
quasiconvergence property of bounded solutions in this space.
In our first result,  the main theorem of \cite{p-Pauthier1}, we proved that
if the limits $\theta^\pm:=u_0(\pm\infty)$ are distinct, then the solution $u$
of \eqref{eq:1}, \eqref{ic:1} is quasiconvergent, if bounded. In
\cite{p-Pauthier2}, we then showed that the same is true
if  $\theta^-=\theta^+=:\theta_0$, and either $f(\theta_0)\ne0$ or
$f(\theta_0)=0$ and $\theta_0$ is a stable
equilibrium of the equation $\dot \xi=f(\xi)$.
In this result, we assumed the following nondegeneracy condition on
the nonlinearity:  
\begin{description}
\item[\bf(ND)] For each $\gamma\in f^{-1}\{0\}$, $f$ is of class $C^1$
 in a neighborhood of $\gamma$ and $f'(\gamma)\neq0.$
\end{description}
Hence, the stability of $\theta_0$ simply means that
$f'(\theta_0)<0$.

In the remaining case, $\theta^\pm=\theta_0$, $f(\theta_0)=0$ with
$\theta_0$ unstable ($f'(\theta_0)>0$), the above quasiconvergence
result is not valid without additional conditions on the solution;
this is  documented by the examples of \cite{P:examples,P:unbal}, as
already mentioned above. It is intriguing, however, 
that all non-quasiconvergent solutions $u$
found in these examples share a prominent feature: 
they are oscillatory in the sense that 
$u(\cdot,t)$ has infinitely many
critical points at all times $t>0$. This raises a 
natural question whether without the oscillations
the solution is necessarily quasiconvergent, if bounded. 
More precisely, the question is
whether the solution of \eqref{eq:1}, \eqref{ic:1} is quasiconvergent,
provided it is bounded and satisfies the
following condition: 
\begin{description}
\item[\bf (NC)] There is $t> 0$ such that $u(\cdot,t)$ has only
  finitely many critical points.
\end{description}

We remark that if (NC) holds for some $t$, then it holds for any
larger $t$ due to well known properties  of the zero number of $u_x(\cdot,t)$
(see Section \ref{sub:zero}). In Remark \ref{rm:suff} below,
we mention some
sufficient conditions for the validity of  (NC) in terms of $u_0$. 

In \cite{p-Pauthier2}, we left open the question whether (NC) alone is
sufficient for the quasiconvergence of $u$; we only proved the
quasiconvergence assuming (NC) holds together with some additional and
somewhat artificial conditions. The main theorem of the present
paper gives a positive answer  without any extra  condition: 

\begin{theorem}\label{thm:1}
  Assume that \emph{(ND)} holds, and 
  $u_0\in \cV$ has both its limits
  $u_0(\pm\infty)$ equal to some $\theta_0\in \R$ with
  $f(\theta_0)=0<f'(\theta_0)$.  Then, if the solution $u$ of \eqref{eq:1},
  \eqref{ic:1} is bounded and satisfies \emph{(NC)},
  it is quasiconvergent: $\om(u)$
  consists of steady states of \eqref{eq:1}. 
\end{theorem}
This theorem, combined with the results of  \cite{p-Pauthier1,p-Pauthier2},
gives the following corollary concerning general bonded solutions which
are nonoscillatory in the spatial variable: 

\begin{corollary}\label{co:1}
  Assume that \emph{(ND)} holds and let $u$ be a bounded solution of
  \eqref{eq:1} such that \emph{(NC)} holds.  Then
   $u$ is  quasiconvergent.
\end{corollary}
\begin{proof}
  Choose a large enough $t_0$ such that (NC) holds with $t=t_0$:
  $u(\cdot,t_0)$ has only
  finitely many critical points. 
  Replacing the initial datum of the solution $u$ by
  $u_0:=u(\cdot,t_0)$, we achieve that
  $u_0$ is monotone near $\pm\infty$;  in
  particular, $u_0\in\cV$. If
  the limits $\theta^\pm:=u_0(\pm\infty)$
  are distinct, or are both equal to  $\theta_0$ where
  either $f(\theta_0)\ne0$ or $\theta_0$ is a stable equilibrium of $\dot
  \xi=f(\xi)$, we apply the results of  \cite{p-Pauthier1} or
  \cite{p-Pauthier2}, respectively. If the limits are both equal to an
  unstable equilibrium of $\dot
  \xi=f(\xi)$, we apply Theorem \ref{thm:1}. We thus obtain the
  quasiconvergence conclusion in all cases. 
\end{proof}

\begin{remark}
  \label{rm:suff}
  {\rm
    \begin{itemize}[align=left,itemindent=3ex,leftmargin=0pt]
    \item[(i)]
    We mention here some simple sufficient conditions, in terms of the
    initial data $u_0$, for the validity of the assumptions on the
    solution $u$ in Theorem \ref{thm:1}.   
    A sufficient condition for the boundedness of the solution of
    \eqref{eq:1}, \eqref{ic:1} is that $u_0$ takes values between two
    constants $\xi<\eta$ satisfying $f(\xi)>0>f(\eta)$. This follows from the
    comparison principle. A sufficient condition for (NC) is that
    $u_0$ has only finitely many critical points if it is of class
    $C^1$. This is a consequence of the
    monotonicity of the zero number of $u_x(\cdot,t)$. More generally,
    (NC) holds if there are constants 
$a<b$ such that the function $u_0$ is monotone and nonconstant on each
of the intervals $(-\infty,a)$, $(b,\infty)$.  Indeed, if this holds,
one shows easily, using the comparison principle (comparing $u$ and
its spatial shifts) that for small $t>0$ the function is strictly
monotone on each of the intervals  $(-\infty,a-1)$, $(b+1,\infty)$;
the strong comparison principle then shows that $u_x(x,t)$ has no zero
in these intervals for small $t$. Consequently, by
properties of the zero number
of $u_x(\cdot,t)$ (cp. Section \ref{sub:zero}),
$u(\cdot,t)$ has only a finite number of critical points for all
$t>0$.
 \item[(ii)] As mentioned above, bounded solutions that do not satisfy
   (NC) are not quasiconvergent in general. In this sense, condition
   (NC) is optimal. However, some generalization are probably still
   possible. For example, one may ask if it is sufficient to
   assume that for some $t$ there is $\rho$ such that $u(\cdot,t)$ has
   no critical points in at least one of the intervals $(-\infty,\rho)$,
   $(\rho,\infty)$. (Note that, as in (i), if
   $u(\cdot,t)$ has no critical points
   in the union of these intervals, then (NC) holds for larger
   times). Another question is whether condition
   (NC) can be replaced by the weaker requirement that
   $u(\cdot,t)-\theta_0$ has only 
   finitely many zeros for some $t$. Our proof does not apply in these
   cases and we do not pursue these generalizations. 
    \end{itemize}
}
\end{remark}

In the proof of Theorem \ref{thm:1}, we build on the strategy
and some technical results of \cite{p-Pauthier2}. The strategy
consists in careful analysis of a certain type of entire solutions of
\eqref{eq:1}. By an entire solution we mean a solution $U(x,t)$ of
\eqref{eq:1} defined for all $t\in\R$ (and $x\in\R$). It is well known
that for any $\vp\in\omega(u)$ there exists a unique entire solution
$U(x,t)$ of \eqref{eq:1} such that
$U(\cdot,0)=\vp$, and this solution satisfies
$U(\cdot,t)\in\omega(u)$ for all $t\in\R$.
This is how entire solutions are relevant for our problem. The
assumption  $u_0\in \cV$ poses some restrictions on the structure of
entire solutions that can possibly be contained in $\om(u)$. 
Using these structural properties  in combination with the chain
recurrence property of $\om(u)$, we were able to prove in
\cite{p-Pauthier2}, assuming that $\theta^\pm=\theta_0$ and 
$f(\theta_0)=0>f'(\theta_0)$, that all the entire solutions in
$\om(u)$ are necessarily steady states. To prove the same in the
present case, 
$\theta^\pm=\theta_0$ and $f(\theta_0)=0<f'(\theta_0)$,
assuming (NC), we need to consider a class of
entire solution not covered by the analysis of \cite{p-Pauthier2}
(see Section \ref{sptraj} below for more details on this).
We prove a classification result for such entire solution (see
Proposition \ref{prop:entire}), after which a general
conclusion from \cite{p-Pauthier2} becomes applicable and we
obtain our quasiconvergent result.

The rest of the paper is organized as follows. In the next section,
we define the concepts of a chain of steady states of \eqref{eq:1}
and spatial trajectories of solutions of \eqref{eq:1}. We use these
concepts to state a proposition which has Theorem \ref{thm:1} as a corollary. 
The proposition  is then
proved in Sections \ref{sec:sptraj} and \ref{sec:entire}. 
In the preliminary Section \ref{prelims}, we recall
several technical results from earlier papers, and discuss the basic
properties of $\alpha$ and $\om$-limit sets and the zero number. 

Below, it will be convenient to assume the following additional
condition on the nonlinearity:
\begin{description}
\item[\bf(MF)]$f$ is globally Lipschitz and there is $\kappa>0$ such
  that for all $s$ with $|s|>\kappa$ one has $ f(s)={s}/{2}.$
\end{description}
Since this condition concerns the 
behavior of $f(u)$ for large values of $|u|$, it can be assumed
with no loss of generality. Indeed, our  quasiconvergence theorem 
deals with an individual bounded solution, thus modifying
$f$ outside the range of this solution has no effect on the
validity of the theorem. 

Conditions (ND), (MF), are our \emph{standing hypotheses on $f$}.
With no loss of generality, shifting $f$ if necessary, we will also
assume that $\theta_0$ in Theorem \ref{thm:1} is equal to
zero. Thus, we henceforth also assume that
\begin{equation}
  \label{eq:17}
  f(0)=0, \quad f'(0)>0.
\end{equation}

\section{Spatial trajectories and chains}
\label{concepts}

As in \cite{p-Pauthier2},  we employ
a geometric technique  involving spatial trajectories of solutions of
\eqref{eq:1}. Our analysis consists mainly in the examination of
how spatial trajectories of entire solutions of \eqref{eq:1} are
related to chains of the planar system corresponding to the equation
 for the steady states of \eqref{eq:1}:
\begin{equation}
  \label{eq:steady}
  u_{xx}+f(u)=0,\quad x\in\R.
\end{equation}

We define the concept of a chain in the next subsection, after
recalling some basic properties of the planar trajectories of
\eqref{eq:steady}. Spatial trajectories of solutions of
entire solutions of \eqref{eq:1} are defined
in Subsection \ref{sptraj}. In that subsection, we
state a result concerning entire solutions which implies
Theorem \ref{thm:1}.

\subsection{Steady states  of \eqref{eq:1}  and chains}\label{stst}
Consider the
planar system
\begin{equation}\label{eq:sys}
  u_x=v,\qquad v_x=-f(u),
\end{equation}
associated with equation \eqref{eq:steady}. 

It  is a Hamiltonian system with respect to the
energy
\begin{equation}\label{energy}
  H(u,v)=\frac{v^2}{2}+F(u),
\end{equation}
where $F(u)=\int_0^uf(s)\,ds$.
Thus, each orbit of \eqref{eq:sys} is
contained in a level set of $H.$ The level sets are symmetric with
respect to the $u-$axis, and our extra hypothesis (MF) implies that
they are all bounded. Therefore, all orbits of \eqref{eq:sys} are bounded
and there are only four types of them: equilibria (all of which are on
the $u-$axis), nonstationary periodic orbits (by which we mean orbits
of nonstationary periodic solutions), homoclinic orbits, and
heteroclinic orbits.
Following a common terminology, we say that a
solution $\vp$ of \eqref{eq:steady} is a \emph{ground state at level
  $\ga$} if the corresponding solution $(\vp,\vp_x)$ of \eqref{eq:sys} is
homoclinic to the equilibrium $(\ga,0)$; we say that $\vp$ is \emph{a
  standing wave of \eqref{eq:1} connecting $\ga_-$ and $\ga_+$} if
$(\vp,\vp_x)$ is a heteroclinic solution of \eqref{eq:sys} with limit
equilibria $(\ga_-,0)$ and $(\ga_+,0)$.

Each nonstationary periodic orbit $\mathcal{O}$ is symmetric about
the $u-$axis and for some $p<q$ one has
\begin{align}
  \mathcal{O}\cap\{ (u,0):u\in\R\} & = \left\{(p,0),(q,0)\right\}, \nonumber \\
  \mathcal{O}\cap \left\{(u,v):v>0\right\} & = \left\{\left( u,\sqrt{2(F(p)-F(u))}\right):u\in(p,q)\right\}. \label{periodicorbits}
\end{align}

Let
\begin{align}
  \mathcal{E} & := \{ (a,0):f(a)=0\}\  \textrm{ (the set of all equilibria of \eqref{eq:sys})}, \nonumber \\
  \mathcal{P}_0 & :=\{(a,b)\in\R^2: (a,b)\textrm{ lies on a nonstationary periodic orbit of \eqref{eq:sys}}\}, \nonumber \\
  \mathcal{P} & := \mathcal{P}_0\cup\mathcal{E}\  \textrm{ (the union of
                all periodic orbits of \eqref{eq:sys}, including the equilibria)}. \nonumber
\end{align}

The following lemma is the same as \cite[Lemma 2.1]{p-Pauthier2}, which, 
except for the last two statements in (i), was originally proved in 
\cite[Lemma 3.1]{Matano_Polacik_CPDE16}. It gives a description
of the phase plane portrait of
\eqref{eq:sys} without the nonstationary periodic orbits.

\begin{lemma}\label{MatPolLemma}
  The following two statements are valid.
  \begin{enumerate}
  \item[(i)] Let $\Sigma$ be a connected component of
    $\R^2\setminus\mathcal{P}_0.$ Then $\Sigma$ is a compact set
    contained in a level set of the Hamiltonian $H$ and one has
    \begin{equation*}
      \Sigma = \left\{(u,v)\in\R^2:u\in J,\ v=\pm\sqrt{2(c-F(u))}\right\}
    \end{equation*}
    where $c$ is the value of $H$ on $\Sigma$ and $J=[p,q]$ for some
    $p,q\in\R$ with $p\leq q.$ Moreover, if $(u,0)\in\Sigma$ and
    $p<u<q,$ then $(u,0)$ is an equilibrium.  If $p<q,$ the points
    $(p,0)$ and $(q,0)$ lie on homoclinic
    orbits. 
    If $p=q,$ then $\Sigma=\{(p,0)\},$ and $p$ is an unstable
    equilibrium of the equation $\dot \xi=f(\xi)$.
  \item[(ii)] Each connected component of the set
    $\R^2\setminus\mathcal{P}$ consists of a single orbit of
    \emph{\eqref{eq:sys}}, either a homoclinic orbit or a heteroclinic
    orbit.
  \end{enumerate}
\end{lemma}

We define a \emph{chain}  as  any connected component of the set
$\R^2\setminus \cP_0$.
Each chain consists of equilibria, homoclinic orbits, and,
possibly, heteroclinic orbits of \eqref{eq:sys}. 
We say that a chain is \emph{trivial} if 
    it consists of a single equilibrium.
    By a \emph{loop} we mean a set $\La\subset \R^2$ which is either the
    union of a homoclinic orbit and its limit equilibrium or the
    union of two heteroclinic orbits, one reflection of the other
    around the $u$ axis,  and their common limit
    equilibria. Obviously, every loop $\La$ is contained in a chain and it
    can be viewed as a Jordan curve in $\R^2$. We denote
    by $\mathcal{I}(\La)$ the interior of $\La$ (the bounded connected 
    component of $\R^2\setminus \La$).
    Similarly we define
    $\mathcal{I}(\cO)$ when $\mathcal{O}$ is a nonstationary periodic
    orbit of \eqref{eq:sys}. If $\Sigma$ is a chain,
    $\mathcal{I}(\Sigma)$ denotes the union of the interiors of the loops
    contained in $\Sigma$. We also define
${\overline{\mathcal{I}}(\Sigma)}=\mathcal{I}(\Sigma)\cup\Sigma.$ The
set ${\overline{\mathcal{I}}(\Sigma)}$ is closed and equal to the
closure of $\mathcal{I}(\Sigma)$, except when $\Sigma$ consists of
a single point, in which case  $\overline{\mathcal{I}}(\Sigma)=\Sigma$. For a
nonstationary periodic orbit $\mathcal{O}$ of \eqref{eq:sys},
${\overline{\mathcal{I}}(\mathcal{O})}$ denotes the closure of  
    $\mathcal{I}(\cO)$.

The following lemma introduces the inner chain and
the outer loop associated with a connected component of
$\mathcal{P}_0$ (see Figure \ref{inandout}). The lemma is
identical with \cite[Lemma 2.2]{p-Pauthier2}. 
 
 \begin{lemma}\label{le:p0}
   Let $\Pi$ be any connected component of $\mathcal{P}_0.$ The
   following statements hold true.
   \begin{enumerate}
   \item[(i)] The set $\Pi$ is open.
   \item[(ii)] There exists a unique chain $\Sigma_{in}$ such that for
     all periodic orbits $\mathcal{O}\subset\Pi$ one has
     \begin{equation*}
       {\overline{\mathcal{I}}\left(\Sigma_{in}\right)}\subset\mathcal{I}(\mathcal{O})\textrm{ and }\mathcal{I}(\mathcal{O})\setminus {\overline{\mathcal{I}}(\Sigma_{in})}\subset\Pi.
     \end{equation*}
   \item[(iii)] If $\Pi$ is bounded, there exists a unique loop
     $\Lambda_{out}$ such that for all periodic orbits
     $\mathcal{O}\subset\Pi$ one has
     \begin{equation*}
       \ol{\mathcal{I}}(\mathcal{O})\subset
       \mathcal{I}(\Lambda_{out}),\textrm{ and }
       \mathcal{I}(\Lambda_{out})\setminus
       {\overline{\mathcal{I}}(\mathcal{O})}\subset\Pi.    
     \end{equation*}
   \item[(iv)] There is a zero $\be$ of $f$ such that $f'(\be)>0$ and
     $(\beta,0)\in\mathcal{I}(\mathcal{O}),$ for all periodic orbits
     $\mathcal{O}\subset\Pi.$
   \item[(v)] If $\mathcal{O}_1,\mathcal{O}_2$ are two distinct
     periodic orbits contained in $\Pi,$ then either
     $\mathcal{O}_1\subset\mathcal{I}\left(\mathcal{O}_2\right)$ or
     $\mathcal{O}_2\subset\mathcal{I}\left(\mathcal{O}_1\right)$
     (thus, $\Pi$ is totally ordered by this relation).
   \end{enumerate}
 \end{lemma}
 We refer to $\Sigma_{in}$ and $\La_{out}$ as the \emph{inner chain
   and outer loop} associated with $\Pi$. If the correspondence to
 $\Pi$ is to be explicitly indicated, we denote them by
 $\Sigma_{in}(\Pi)$ and $\La_{out}(\Pi)$, respectively.

\begin{figure}[h]
  \centering
  \includegraphics[scale=.6]{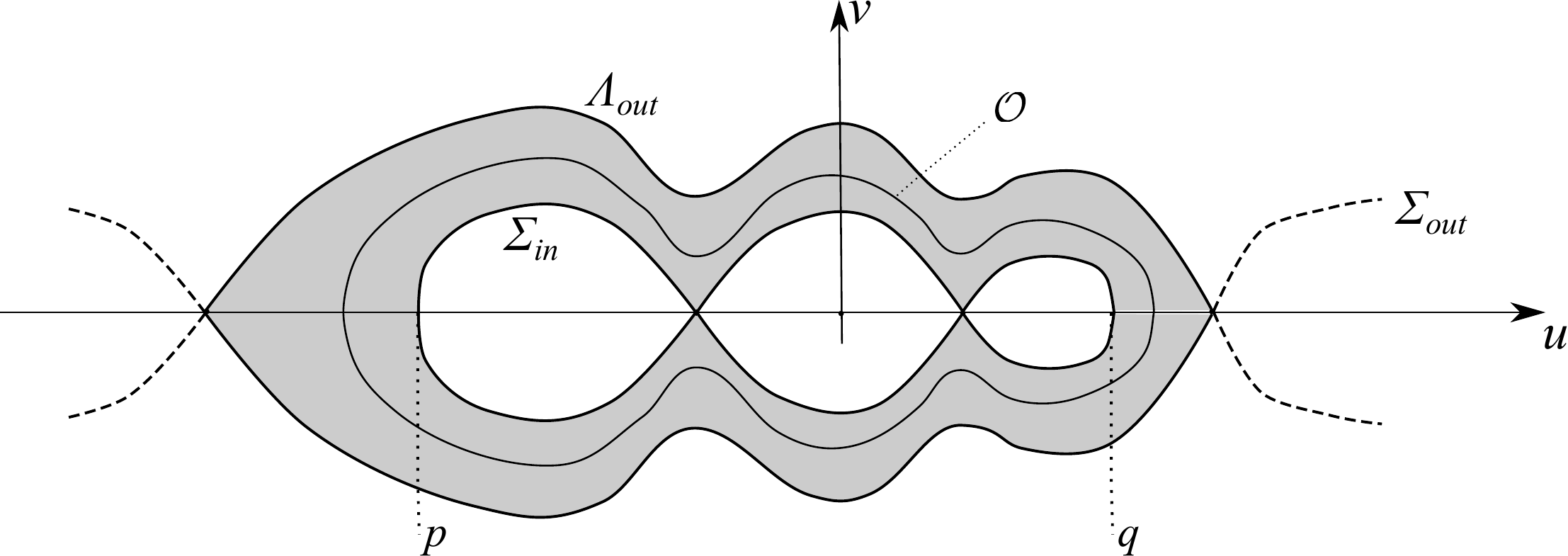}
  \caption[The inner chain and outer loop]{The inner chain and outer loop
    associated with a connected component $\Pi$ of $\cP_0$:
    $\La_{out}$ and $\Sigma_{in}$ form
    the boundary of $\Pi$.  The outer loop can be a heteroclinic loop
    (as in this figure) or a homoclinic loop, and it is part of a
    chain $\Sigma_{out}$. The points $p$ and $q$ are as in Lemma
    \ref{MatPolLemma} for $\Sigma=\Sigma_{in}$.
    \label{inandout}}
\end{figure}

Below, the connected
component of $\mathcal{P}_0$ whose closure contains $(0,0)$ will play
a prominent role. We denote it by  $\Pi_0$.  Note that
$\Pi_0$ is well defined, for $f'(0)>0$ implies that $(0,0)$ is a
center for \eqref{eq:sys}, which is to say that it  has a neighborhood
foliated by periodic orbits. 
 
\begin{figure}[ht]
  \includegraphics[scale=.7]{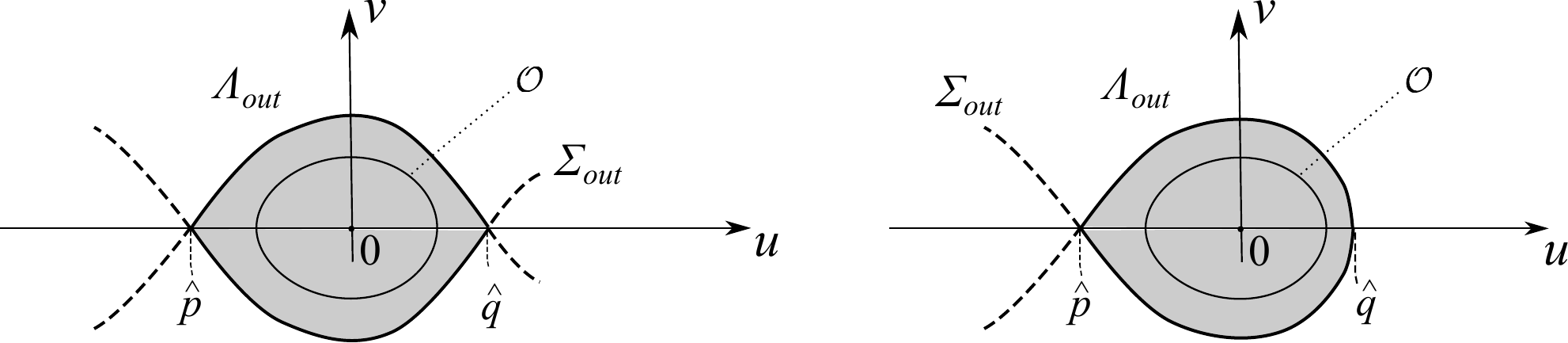}
 \centering
  \caption[Component $\Pi_0$]{The shaded region indicates the
    connected component $\Pi_0$
    containing the point $(0,0)$. The corresponding outer
    loop $\La_{out}$ is a heteroclinic loop in the figure on the left
    and a homoclinic loop in the figure of the right. Points
    $\hat p$ and $\hat q$ indicate the intersections of  $\La_{out}$
    with the $u$-axis.
    The inner chain is trivial: $\Sigma_{in}=\{(0,0)\}$.    
    \label{pi0}}
\end{figure}

\subsection{A key result on spatial trajectories of  entire solutions}
\label{sptraj} 

In this subsection, we introduce
spatial trajectories of entire solutions of \eqref{eq:1}. As we
explain, Theorem \ref{thm:1} follows from a result on
entire solutions stated in Proposition \ref{prop:missing} below.

 For any $\varphi \in C^1(\R)$,  we define
\begin{equation}
  \label{eq:straj}
  \tau(\vp):=\left\{ \left( \vp(x),\vp_x(x)\right):x\in\R\right\} 
\end{equation}
and refer to this set as the \textit{spatial trajectory (or orbit)} of
$\vp.$  Note that if $\vp$ is a steady state of \eqref{eq:1}, then $\tau(\vp)$
is the usual trajectory of the solution $(\varphi,\varphi_x)$ of the
planar system \eqref{eq:sys}. If $U$ is an
entire solution of \eqref{eq:1}; we refer to the
collection $\tau(U(\cdot,t))$, $t\in\R$, as the spatial trajectories
of $U$. 

If $Y\subset C^1(\R)$, $\tau(Y)\subset \R^2$ is the union of the spatial
trajectories of the functions in $Y$:
\begin{equation}
  \label{eq:2}
  \tau(Y):=\left\{ \left( \vp(x),\vp_x(x)\right):x\in\R,\,\varphi\in Y\right\}. 
\end{equation}

Assume now that $u_0\in C_b(\R)$, $u_0(\pm \infty)=0$
(recalling that relations \eqref{eq:17} are assumed to hold),
and the solution $u$ of \eqref{eq:1},
\eqref{ic:1} is bounded. For a description of $\om(u)$,
some results  relating the spatial trajectories
of  entire solutions of
\eqref{eq:1} to the chains of \eqref{eq:sys} are crucial. 
If one can prove that, for an entire solution $U$,  
the spatial trajectories $\tau(U(\cdot,t))$, $t\in\R$, are all
contained in a chain, then  a unique-continuation type result
shows that $U$ is a steady state of \eqref{eq:1}
(see Lemma \ref{le:2.7} below).
Thus, the  quasiconvergence of  $u$ can be proved
by showing that  $\tau(\om(u))$ is contained in a chain.  
We now explain a key idea of how this can be accomplished.

Let us first scrutinize the possibility that
for some entire solution $U$ with $U(\cdot,t)\in\om(u)$
a spatial trajectory $\tau(U(\cdot,t_0))$
is not contained in any chain for some $t_0\in\R$. 
It was proved in \cite[Proposition 3.2]{p-Pauthier2} that then
none of the trajectories $\tau(U(\cdot,t))$, $t\in\R,$ can intersect any chain. 
 This clearly implies that there is a connected component $\Pi$ of
$\mathcal{P}_0$ such that 
\begin{equation}
      \label{eq:5}
      \underset{t\in\R}{{\textstyle \bigcup}}\tau\left(
        U(\cdot,t)\right)\subset\Pi.
    \end{equation}
    The connected component $\Pi$  has to be bounded as also shown in
    \cite{p-Pauthier2}.

Trying to rule  \eqref{eq:5} out,  we look for a contradiction.  
We consider the $\om$ and $\al$-limit sets of the entire solution 
$U$, denoted by $\om(U)$, $\al(U)$, respectively;
$\om(U)$ is defined as in \eqref{defomega} and the definition of
$\al(U)$ is analogous, with $t_n\to\infty$ replaced by
$t_n\to-\infty$. 
Take the inner chain $\Sigma_{in}(\Pi)$ 
    and the outer loop  $\Lambda_{out}(\Pi)$ associated
    with $\Pi$, as in Lemma \ref{le:p0}.
As in   \cite[Section 6]{p-Pauthier2}, a contradiction is obtained 
if the following relations can be derived from \eqref{eq:5}: 
    \begin{equation}
      \label{eq:6}
      \tau\left(\alpha(U)\right)\subset\Sigma_{in}(\Pi),\qquad \tau\left(\omega(U)\right)\subset\Lambda_{out}(\Pi).
    \end{equation}
The reason why \eqref{eq:6} leads to a contradiction can intuitively be
explained as follows.  Relations \eqref{eq:6} show that there is 
a specific    ``direction'' of the flow of \eqref{eq:1}
in $\om(u)$: 
 $U(\cdot,t)$ always goes \emph{from the inner chain to
the outer loop} as $t$ increases from $-\infty$ to $\infty$.
However, the existence of such a flow
direction is inconsistent with well-known chain-recurrence properties of
the $\omega$-limit sets and thus the contradiction
(see   \cite[Section 6]{p-Pauthier2} for details).

Under the assumptions (ND), (MF), and \eqref{eq:17},  it has been
proved in \cite{p-Pauthier2} that relations \eqref{eq:6}
do follow from \eqref{eq:5} for any connected component $\Pi$ of
$\cP_0$, with the notable exception of $\Pi=\Pi_0$. Recall that $\Pi_0$ is the
connected component whose closure contains $(0,0)$;  in this case,
$\Sigma_{in}(\Pi_0)$ is the trivial chain $\{(0,0)\}$.
As noted in
\cite[Remark 6.3]{p-Pauthier2}, the lack of \eqref{eq:6} in the case
$\Pi=\Pi_0$ was the only reason why we could not give
a general quasiconvergence theorem in the case $u_0(\pm \infty)=0$ and
$f(0)=0<f'(0)$. In the present paper, we provide the proof of
\eqref{eq:6} in the case $\Pi=\Pi_0$ and thereby prove Theorem \ref{thm:1}.

For reference, we state here the result which implies Theorem \ref{thm:1},
as explained above.    
\begin{prop}\label{prop:missing} Assuming 
 \emph{ (ND), (MF)}, and \eqref{eq:17}, let $u_0\in \cV$ be
  a function satisfying 
  $u_0(\pm\infty)=0$ such that the solution $u$ of \eqref{eq:1},
  \eqref{ic:1} is bounded and \emph{(NC)} holds.
  Let $U$ be an entire solution of
  \eqref{eq:1} such that $U(\cdot,t)\in 
  \om(u)$ for all $t\in\R$. If  \eqref{eq:5} holds with $\Pi=\Pi_0$,
  then 
   \begin{equation}
      \label{eq:6a}
      \alpha(U)=\{0\},\qquad
      \tau\left(\omega(U)\right)\subset\Lambda_{out}(\Pi_0). 
    \end{equation}  
  \end{prop}
  Of course, $ \alpha(U)=\{0\}$ is equivalent to
  $\tau\left(\alpha(U)\right)=\{(0,0)\}$,
  so \eqref{eq:6} and \eqref{eq:6a} are the same
  statements when $\Pi=\Pi_0$ (and $\Sigma_{in}(\Pi_0)=\{(0,0)\}$).

  The case $\Pi=\Pi_0$
  differs from the case when
  $\Sigma_{in}(\Pi)$ is a nontrivial chain
  in several aspects. One key difference is
that we need to take into account
the possibility that the spatial limits $U(\pm\infty,t)$
of the entire solution $U$ depend on $t$ (this
can be ruled out easily if $\Sigma_{in}(\Pi)$ is a nontrivial chain,
see \cite[Lemma 3.9]{p-Pauthier2}). Even  
when $U(\pm\infty,t)$ are independent of $t$,
the case when one or both of them is equal to 0,
an unstable equilibrium of $\dot\xi=f(\xi)$, is not encountered in the case  
$\Pi\ne \Pi_0$ (the limits $U(\pm\infty,t)$ are always equal to a
stable equilibrium of $\dot\xi=f(\xi)$ if $\Sigma_{in}(\Pi)$
is a nontrivial chain).
On the other hand, 
assumption (NC) has some consequences on the structure of relevant
entire solutions (see Section \ref{sec:sptraj}), which we exploit
in the proof of \eqref{eq:6}.

\section{Preliminaries}\label{prelims}
In this section, we first recall basic  properties of the
zero-number functional and 
various limit sets of bounded solutions of \eqref{eq:1} and then state
some  results from earlier paper that will be referred to in the proof of
Proposition \ref{prop:missing}.

\subsection{Zero number}\label{sub:zero}

Consider a linear parabolic
equation
\begin{equation}\label{eq:lin}
  v_t=v_{xx}+c(x,t)v,\qquad x\in\R,\ t\in\left( s,T\right),
\end{equation}
where $-\infty\leq s<T\leq \infty$ and $c$ is a bounded measurable
function. Note that whenever
$u$, $\bar u$ are bounded solutions of
\eqref{eq:1},  their difference $v=u-\bar u$ satisfies
\eqref{eq:lin} with a suitable function $c$.  Similarly, $v=u_x$ and
$v=u_t$ are solutions of such a linear equation.

For an interval $I=(a,b),$ with $-\infty\leq a < b\leq \infty,$ we
denote by $z_I(v(\cdot,t))$ the number, possibly infinite, of zeros
$x\in I$ (counted without their multiplicities)
of the function $x\mapsto v(x,t).$ If $I=\R$ we usually omit
the subscript $\R$:
$$
z(v(\cdot,t)):=z_\R(v(\cdot,t)).
$$
The following intersection-comparison principle holds (see
\cite{Angenent_Crelle88,Chen_MathAnn98}).
\begin{lemma}\label{le:zero}
  Let $v$ be a nontrivial solution of \eqref{eq:lin} and
  $I=(a,b),$ with $-\infty\leq a < b\leq \infty.$ Assume that the
  following conditions are satisfied:
  \begin{itemize}
  \item if $b<\infty,$ then $v(b,t)\neq0$ for all
    $t\in\left( s,T\right),$
  \item if $a>-\infty,$ then $v(a,t)\neq0$ for all
    $t\in\left( s,T\right).$
  \end{itemize}
  Then the following statements hold true.
  \begin{enumerate}
  \item[(i)] For each $t\in\left( s,T\right),$ all zeros of
    $v(\cdot,t)$ are isolated. In particular, if $I$ is bounded, then
    $z_I(v(\cdot,t))<\infty$ for all $t\in\left( s,T\right).$
  \item[(ii)] The function $t\mapsto z_I(v(\cdot,t))$ is monotone
    nonincreasing on $(s,T)$ with values in
    $\N\cup\{0\}\cup\{\infty\}.$
  \item[(iii)] If for some $t_0\in(s,T)$ the function $v(\cdot,t_0)$
    has a multiple zero in $I$ and $z_I(v(\cdot,t_0))<\infty,$ then
    for any $t_1,t_2\in(s,T)$ with $t_1<t_0<t_2,$ one has
    \begin{equation}\label{zerodrop}
      z_I(v(\cdot,t_1))>z_I(v(\cdot,t_0))\ge z_I(v(\cdot,t_2)).
    \end{equation}
  \end{enumerate}
\end{lemma}
If \eqref{zerodrop} holds, we say that $z_I(v(\cdot,t))$ drops at
$t_0$. 

We will also use a version of Lemma \ref{le:zero} for
time-dependent intervals; it  is derived easily from Lemma \ref{le:zero}
(cp.  \cite[Section~2]{p-B-Q}).

\begin{lemma}
  \label{le:zerot}
  Let $v$ be a nontrivial solution of \eqref{eq:lin} and
  $I(t)=(a(t),b(t))$, where $-\infty \le a(t)<b(t)\le \infty$ for
  $t\in(s,T)$.  Assume that the following conditions are satisfied:
  \begin{itemize}
  \item[ {\rm (c1)}] Either $b\equiv \infty$ or $b$ is a (finite)
    continuous function on (s,T).  In the latter case,
    $v(b(t),t)\ne 0$ for all $t\in(s,T)$.
  \item[\rm (c2)] Either $a\equiv -\infty$ or $a$ is a continuous
    function on (s,T).  In the latter case, $v(a(t),t)\ne 0$ for all
    $t\in(s,T)$.
  \end{itemize}
  Then statements (i), (ii) of Lemma \ref{le:zero} are valid with $I$,
  $a$, $b$ replaced by $I(t)$, $a(t)$, $b(t)$, respectively; and
  statement (iii) of Lemma \ref{le:zero} is valid with all occurrences
  of $z_I(v(\codt,t_j))$, $j=0,1,2$, replaced by
  $z_{I(t_j)}(v(\codt,t_j))$, $j=0,1,2$, respectively.
\end{lemma}

The next lemma is a robustness result of  \cite{Du_Matano}.
\begin{lemma}\label{le:robust}
  Let $w_n(x,t)$ be a sequence of functions converging to $w(x,t)$ in
  $\displaystyle C^1\left( I\times(s,T)\right)$ where $I$ is an open
  interval. Assume that $w(x,t)$ solves a linear equation
  \eqref{eq:lin}, $w\not\equiv0$, and $w(\cdot,t)$ has a multiple zero
  $x_0\in I$ for some $t_0\in(s,T)$.  Then there exist sequences
  $x_n\to x_0$, $t_n\to t_0$ such that for all sufficiently large $n$
  the function $w_n(\cdot,t_n)$ has a multiple zero at $x_n$.
\end{lemma}

\subsection{Limit sets and entire solutions}\label{sec:inv}

The $\omega-$limit set of a bounded solution $u$ of
\eqref{eq:1} is defined as in \eqref{defomega}, with the
convergence in $L^\infty_{loc}(\R)$.  As already noted above,
it is a nonempty, compact, and connected set in  $L^\infty_{loc}(\R)$.
It is also well known that $\om(u)$
has the following invariance property: for any
$\vp\in\omega(u),$ there is an entire solution $U(x,t)$ of
\eqref{eq:1} such that
\begin{equation}\label{entiresol}
  U(\cdot,0)=\vp,\qquad U(\cdot,t)\in\omega(u)\quad (t\in\R).
\end{equation}
In fact, if a sequence $t_n\to\infty$ is such that
$u(\cdot,t_n)\to \vp$ in $L^\infty_{loc}(\R)$, then, for a
subsequence, we have  $u(\cdot,t_n+\cdot)\to U$ in $C^1_{loc}(\R^2)$,
where $U$ is an entire solution of \eqref{eq:1} satisfying
\eqref{entiresol}. This follows by compactness arguments based on
parabolic estimates  (see \cite[Section
3.2]{p-Pauthier2} for more details.) 
Note that the entire solution satisfying
$U(\cdot,0)=\vp$  is uniquely determined by $\varphi$; this
follows from the uniqueness and backward uniqueness for the Cauchy
problem \eqref{eq:1}, \eqref{ic:1}.

The above considerations also imply that  $\omega(u)$ 
is unaffected if the convergence in \eqref{defomega} is taken in
$C_{loc}^1(\R)$, rather than in $L^\infty_{loc}(\R)$, and therefore
$\omega(u)$ is connected in $C_{loc}^1(\R)$ as well.
Hence, the set
\[
\tau\left(\omega(u)\right)=\left\{
  (\vp(x),\vp_x(x)):\vp\in\omega(u),x\in\R\right\} =
\underset{\vp\in\omega(u)}{\textstyle \bigcup}\tau(\vp)
\]
is connected in $\R^2$. (Here, $\tau(\varphi)$ is as in
\eqref{eq:straj}.)  Also, obviously, $\tau(\vp)$ is connected in
$\R^2$ for all $\vp\in\omega(u).$

If $U$ is a bounded entire solution of \eqref{eq:1}, we define its
$\alpha-$limit set by
\begin{equation}\label{defalpha}
  \alpha(U):=\left\{ \vp\in C_b(\R):U(\cdot,t_n)\to\vp\textrm{ for some sequence }t_n\to-\infty\right\}.
\end{equation}
Here, again, the convergence is in $L_{loc}^\infty(\R)$, but due to
parabolic regularity, it can be taken in $C_{loc}^1(\R)$ with no
effect on $\al(U)$. The
$\alpha$-limit set has similar properties as the $\omega-$limit set:
it is nonempty, compact and connected in $L^\infty_{loc}(\R)$ as well
as in $C^1_{loc}(\R)$, and for any $\varphi\in \al(U)$ there is an
entire solution $\tilde U$ such that $\tilde U(\cdot,0)=\vp$ and
$\tilde  U(\cdot,t)\in\al(U)$ for all $t\in\R$.  The connectivity property of
$\al(U)$ implies that the set
$$
\tau\left(\alpha(U)\right)=\left\{
  (\vp(x),\vp_x(x)):\vp\in\alpha(U),x\in\R\right\} =
\underset{\vp\in\alpha(U)}{\textstyle\bigcup}\tau(\vp)
$$
is connected in $\R^2.$

For a bounded entire solution $U$ of \eqref{eq:1}, we define
generalized notions of $\al$ and $\om$-limit
sets as follows: 
\begin{align}
  \Omega(U) & := \left\{ \vp\in C_b(\R):U(\cdot+x_n,t_n)\to\vp\textrm{
              for some sequences }x_n\in\R, \ t_n\to\infty\right\} \label{defOmega}, \\
  A(U) & := \left\{ \vp\in C_b(\R):U(\cdot+x_n,t_n)\to\vp\textrm{ for some sequences }x_n\in\R, \ t_n\to-\infty\right\}. \label{defAlpha}
\end{align}
The convergence can be taken in $L_{loc}^\infty(\R)$ or
$C^1_{loc}(\R)$ without altering the sets $\Om(U)$,
$A(U)$.  Both these sets are nonempty, compact and connected in
$C^1_{loc}(\R)$, and they have a similar invariance property as
$\om(U)$, $\al(U)$.   Also, by their definitions, the
sets $\Om(U)$, $A(U)$ are translation invariant as well. Further, the
definitions and parabolic regularity imply that the sets
\begin{equation*}
  \tau\left( A(U)\right)= \underset{\vp\in A(U)}{\textstyle\bigcup}\tau(\vp),
  \quad\tau\left(\Omega(U)\right) =\underset{\vp\in\Omega(U)}{\textstyle\bigcup}\tau(\vp)
\end{equation*}
are connected and compact in $\R^2.$ We remark that the sets
$\tau(\om(u))$, $\tau(\al(u))$ are both connected (as noted above),
but they are not necessarily compact in $\R^2$.

\subsection{Further  technical  results}\label{sec:24}

Throughout this subsection, we assume that $u_0$ is in  $C_b(\R)$ (not
necessarily in $\cV$), $u$ is the solution of \eqref{eq:1}, \eqref{ic:1}
and it is bounded.

In view of the invariance property of $\om(u)$ (see
\eqref{entiresol}), the following lemma gives a criterion for an
element $\varphi\in \om(u)$ to be a steady state. This
unique-continuation
type result is proved in a more general form in
\cite[Lemma 6.10]{Polacik_terrasse}.
\begin{lemma}\label{le:2.7}
  Let $\varphi:=U(\codt,0)$, where $U$ is a solution of \eqref{eq:1}
  defined on a time interval $(-\de,\de)$ with $\de>0$ (this holds in
  particular if $\varphi\in \om(u)$).  If $\tau(\varphi)\subset\Sigma$
  for some chain $\Sigma,$ then $\varphi$ is a steady state of
  \eqref{eq:1}.
\end{lemma}

As already noted above, it is proved in \cite{Gallay-S} (see also
\cite{Gallay-S2}) that the $\om$-limit set of any bounded solution of
\eqref{eq:1} contains a steady state. We will use this result for
entire solutions: 

\begin{theorem}
  \label{thm:GS} If $U$ is a bounded entire solution of \eqref{eq:1},
  then each of the sets $\om(U)$ and $\al(U)$ contains a steady state
  of \eqref{eq:1}.
\end{theorem}
The result concerning the  $\al$-limit set follows from the result of
the $\om$-limit set via compactness and invariance properties of
$\al(U)$:  taking an entire solution
$\tilde U$ with $\tilde U(\cdot,t)\in\al(U)$ for all $t$, we have $\om(\tilde
U)\subset \al(U)$ and $\om(\tilde U)$ contains a steady state.

The following lemma is essentially the same as 
\cite[Lemma 2.11]{p-Pauthier2} (see also \cite[Proof
of Proposition 2.1]{p-B-Q}).  The only difference is that
in \cite[Lemma 2.11]{p-Pauthier2}, $\theta$ is a constant 
whereas here we allow $\theta=\theta(t)$ to continuously depend on
$t$. This makes just a notational difference in the proof given in
\cite[Lemma 2.11]{p-Pauthier2}. 

\begin{lemma}\label{le:BPQ}
  Let $U$ be a solution of \eqref{eq:1} on $\R\times J$, where
  $J\subset \R$ is an open time interval, and let 
  $\theta(t)$ be a continuous real function on $J$.
  Assume that for each $t\in J$
  the function $U(\cdot,t)-\theta(t)$ has at least
  one zero and
 $$
 \xi(t):=\sup\{ x:U(x,t)=\theta(t)\}
 $$
 is finite and depends continuously on $t\in J$.  Then, for any
 $t_0, t_1\in J$ satisfying the relations $t_1>t_0$ and
 $\xi(t_1)<\xi(t_0)$, the function $U_x(\cdot,t_1)$ is of constant
 sign on the interval $(\xi(t_1),\xi(t_0)]$.  If $J=(-\infty,b)$ for
 some $-\infty<b\le \infty$ and $\limsup_{t\to-\infty}\xi(t)=\infty$,
 then $U_x$ is of constant sign on $(\xi(t),\infty),$ for all
 $t\in J.$

 Analogous statements hold for $\xi(t)=\inf\{x:U(x,t)=\theta(t)\}.$
\end{lemma}

We next state a quasiconvergence result from 
\cite{p-Pauthier1} (cp. \cite[Theorem 2.12]{p-Pauthier2}).
For any $\lambda\in\R$, consider the function $V_\la u$ 
defined by
\begin{equation}\label{reflexion}
  V_\la u(x,t)=u(2\la-x,t)-u(x,t),\quad x\in\R,\,t\ge 0.
\end{equation}
\begin{theorem}\label{thmPP1}
  Assume that $u_0\in\cV$ and one of the following conditions holds:
  \begin{itemize}
  \item[(i)] $u_0(-\infty)\ne u_0(\infty)$,
  \item[(ii)] there is $t>0$ such that for all $\la\in\R,$ one has
    $z(V_\la u(\cdot,t))<\infty.$
  \end{itemize}
  Then, $u$ is quasiconvergent. Moreover, $\omega(u)$ does not contain any nonconstant periodic function.
\end{theorem}

The following result is the same as \cite[Lemma 2.13]{p-Pauthier2}.
It 
is a variant of the Squeezing Lemma from
\cite{P:entire}. 
\begin{lemma}\label{le:squeezing}
  Let $U$ be a bounded entire solution of \eqref{eq:1} such that if
  $\beta\in f^{-1}\{0\}$ is an unstable equilibrium of the equation
  $\dot \xi=f(\xi)$, 
  then
  \begin{equation}\label{eq:2.15}
    z\left( U(\cdot,t)-\beta\right)\leq N\quad (t\in\R)
  \end{equation}
  for some $N<\infty.$ Let $K$ be any one of the following subsets of
  $\R^2:$
 $$
 \underset{t\in\R}{{\textstyle \bigcup}}\tau\left( U(\cdot,t)\right),\quad
 \tau\left(\omega(U)\right), \quad \tau\left(\Omega(U)\right),\quad
 \tau\left(\alpha(U)\right), \quad \tau\left( A(U)\right).
 $$
 Assume that $\mathcal{O}$ is a  nonstationary periodic orbit of
 \eqref{eq:sys} such that one of the following inclusions holds:
 \begin{equation*}
   \text{(i) \quad $K\subset   \mathcal{I}(\mathcal{O})$,\qquad\qquad
   (ii)\quad $K\subset\R^2\setminus{\overline{\mathcal{I}}(\mathcal{O})}$.}
 \end{equation*}
 Let
 $\Pi$ be the connected component of $\cP_0$ containing $\cO$.
 If (i) holds, then 
 $K\subset{\overline{\mathcal{I}}\left(\Sigma_{in}(\Pi)\right)}$;
 and if (ii) holds, then
 $K\subset\R^2\setminus\mathcal{I}\left(\Lambda_{out}(\Pi)\right)$
 (in particular, $\Pi$ is necessarily bounded in this case).
\end{lemma}

Finally, we recall the following well known result concerning the
solutions in $\cV$ (the proof can be found in \cite[Theorem
5.5.2]{Volpert3}, for example).
\begin{lemma}\label{le:limits}
  Assume that $u_0\in \cV$. Then the limits
  \begin{equation}\label{limits-t}
    \theta_-(t):=\lim_{x\to-\infty}u(x,t),\qquad \theta_+(t):=\lim_{x\to\infty}u(x,t)
  \end{equation}
  exist for all $t> 0$ and are solutions of the following
  initial-value problems:
  \begin{equation}\label{eq:3.2}
    \dot{\theta}_\pm=f(\theta_\pm),\qquad \theta_\pm(0)=u_0(\pm\infty).
  \end{equation}
\end{lemma}

\section{Entire solutions in
  $\omega(u)$}\label{sec:sptraj}
Throughout this section, we assume---in addition to the standing
hypotheses (ND), (MF), and \eqref{eq:17}---
that $u_0\in \cV$, $u_0(\pm\infty)=0$, and the solution of
\eqref{eq:1}, \eqref{ic:1} is bounded and satisfies (NC). We
reserve the symbol $u(x,t)$ for this fixed solution. 

In the following lemma, we derive some consequences of the
assumption (NC) concerning entire solutions $U(\cdot,t)$ in $\om(u)$.

\begin{lemma}
  \label{le:entirein0} Let $U$
  be an entire solution of \eqref{eq:1} such that $U(\cdot,t)\in 
  \om(u)$ for all $t\in\R$. Then either $U_x\equiv 0$ or $U$ has the following properties:
  \begin{itemize}
  \item[(i)] Each of the functions $U(\cdot,t)$, $U_x(\cdot,t)$ has
    only finitely many zeros, all of them simple, and the number of
    these zeros is bounded by a constant independent of $t$.
    \item[(ii)] For each $t\in \R$, the following limits exist:
      \begin{equation}\label{eq:29a}
    \Theta_-(t):=\underset{x\to-\infty}{\lim}U(x,t),\quad
    \Theta_+(t):=\underset{x\to\infty}{\lim}U(x,t).
  \end{equation}
   \item[(iii)] For each $t\in \R$, the spatial trajectory
     $\tau(U(\cdot,t))$ is a simple curve in $\R^2$, that is, it has
     no self-intersections.
   \item[(iv)]  If \eqref{eq:5} holds with $\Pi=\Pi_0$, then,
     for any $t\in \R$, the function $U(\cdot,t)$ has no
    positive local minima and no negative local maxima. 
  \end{itemize}
\end{lemma}
\begin{proof} Assume that $U_x\not\equiv 0$.

  We know (cp. Section \ref{sec:inv}) that for some sequence
  $t_n\to\infty$  we have
  \begin{equation}
      \label{eq:54}
      u(\cdot,\cdot+t_n)\to U, \quad u_x(\cdot,\cdot+t_n)\to U_x
    \end{equation}
    with the convergence in $L_{loc}^\infty(\R^2)$ in both cases.
    Moreover, since $f$ is Lipschitz, the function
    $u_x$ is bounded in $C^{1+\al}(\R\times[1,+\infty))$ for some
    $\al\in(0,1)$. Therefore, possibly after $\{t_n\}$ is replaced by
    a subsequence, the second convergence in \eqref{eq:54} takes  place in
    $C_{loc}^1(\R^2)$ as well.

  We now prove that all zeros of
    $U_x(\cdot,t)$ are simple. Suppose for a contradiction that $x_0$
    is multiple zero of  $U_x(\cdot,t_0)$  for some $t_0$.
   It then follows from  \eqref{eq:54} and Lemma
    \ref{le:robust}, that there is a sequence $\tau_n\to0$ such that
    $u_x(\cdot,\cdot+ t_n+\tau_n)$ has a multiple
    zero. Thus, by Lemma \ref{le:zero}, the zero number
    $z(u_x(\cdot,t))$ drops at $t=t_n+\tau_n$. Since $
    t_n+\tau_n\to\infty$,  the monotonicity of
    the zero number gives  
    $z(u_x(\cdot,t))=\infty$ for all $t>0$, in contradiction to (NC).
    The contradiction proves that the zeros of $U_x(\cdot,t)$ are
    indeed simple for any $t\in\R$.

    From  \eqref{eq:54} and (NC),
    it  now follows that  $z(U_x(\cdot,t))$ is bounded by a constant
    independent of  $t$. Consequently, by the mean value theorem,
    $z(U(\cdot,t))$ is also bounded by a constant
    independent of  $t$. The proof of statement (i) is complete.

    Statement (i) implies that for each $t$ the (bounded) function
    $U(\cdot,t)$ is monotone near $\pm\infty$. This gives (ii).

    To prove statement (iii), we go by contradiction. Suppose that for
    some 
    $t_0\in\R$ the curve $\tau(U(\cdot,t_0))$ is not simple:
    there are  $x_0,\eta\in\R$ with $\eta\ne 0$, such that
    \begin{equation}
      \label{eq:47}
      (U(x_0,t_0),U_x(x_0,t_0)) = (U(x_0+\eta,t_0),U_x(x_0+\eta,t_0)).
    \end{equation}
Consider the function $v(x,t):=U(x+\eta,t)-U(x,t)$. It is a solution
of a linear parabolic equation \eqref{eq:lin}, and \eqref{eq:47} means
that $x_0$ is a multiple zero of $v(\cdot,t_0)$. By \eqref{eq:54} and Lemma
\ref{le:robust}, there is a sequence $t_n\to\infty$ such that the
function $u(\cdot+\eta,t_n)-u(\cdot,t_n)$ has a multiple zero (near
$x_0$). Therefore, the same arguments  as the ones used above
for the function $u_x$ show that  for all $t>0$ one has
$z(u(\cdot+\eta,t)-u(\cdot,t))=\infty$.  
By Lemma \ref{le:zero}, all zeros of $u(\cdot+\eta,t)-u(\cdot,t)$ are
isolated for $t>0$. Therefore, given any $t>0$, 
 there is a sequence $x_n$ with $|x_n|\to\infty$ such that 
$u(x_n+\eta,t)-u(x_n,t)=0$. Now, for each $n$, the function $u(\cdot,t)$
has a critical point between $x_n$ and $x_n+\eta$, so it has
infinitely many critical points. This contradiction to condition
(NC) proves statement (iii).

For statement (iv), we refer the reader to \cite[Lemma 3.11(ii)]{p-Pauthier2}. 
\end{proof}

\section{Entire solutions with spatial
  trajectories in  $\Pi_0$}
\label{sec:entire}
In this section,
we investigate entire solutions $U$
of \eqref{eq:1} such that 
\begin{itemize}
\item[(c0)] \begin{equation}
      \label{eq:5a}
      \underset{t\in\R}{{\textstyle \bigcup}}\tau\left(
        U(\cdot,t)\right)\subset\Pi_0. 
    \end{equation}
  \end{itemize}
Moreover, we will assume that if $U_x\not\equiv 0$, then $U$
satisfies the following conditions:
\begin{itemize}
  \item[(ci)] There is an integer $m$ such that
    \begin{equation}
      \label{eq:59}
      z(U(\cdot,t)),\ z(U_x(\cdot,t))\le m\quad (t\in \R)
    \end{equation}
    and all zeros of 
    $U(\cdot,t)$ and $U_x(\cdot,t)$ are simple.
    \item[(cii)] The following limits exist
      \begin{equation}\label{eq:29b}
    \Theta_-(t):=\underset{x\to-\infty}{\lim}U(x,t),\quad
    \Theta_+(t):=\underset{x\to\infty}{\lim}U(x,t).
  \end{equation}
   \item[(ciii)] For each $t\in \R$, the spatial trajectory
     $\tau(U(\cdot,t))$ is a simple curve in $\R^2$ (it has
     no self-intersections).
   \item[(civ)]  For any $t\in \R$, the function $U(\cdot,t)$ has no
    positive local minima and no negative local maxima. 
  \end{itemize}  

As shown in Lemma \ref{le:entirein0}, if the entire solution $U$
satisfying (c0) comes from an $\omega$-limit set:
$U(\cdot,t)\in \om(u)$, where $u$ is as in Section \ref{sec:sptraj},
then either $U_x\equiv 0$ or (ci)--(civ) hold.
Therefore the following
proposition, which is the main result of this section,
implies Proposition \ref{prop:missing} (and thereby Theorem
\ref{thm:1}).   
\begin{proposition}
  \label{prop:entire}
  Let $U$ be an entire solution satisfying \emph{(c0)} such that
  either $U_x\equiv 0$ or conditions \emph{(ci)--(civ)} hold. Then
  \begin{equation}
      \label{eq:6b}
      \alpha(U)=\{0\},\qquad
      \tau\left(\omega(U)\right)\subset\Lambda_{out}(\Pi_0). 
    \end{equation}      
\end{proposition}

Note that, by Lemma \ref{le:2.7}, the second
inclusion means that $\omega(U)$
consists of steady states of \eqref{eq:1} whose trajectories are
contained in $\La_{out}(\Pi_0)$.  Thus the proposition says that
any entire solution $U$ with the indicated properties is a
connection, in $L^\infty_{loc}(\R)$, from 0 to  a set of
steady states with trajectories in the outer loop.
In the process of proof of the proposition, we will make this
conclusion more precise in some cases.

Although our main purpose of investigating entire solutions $U$
satisfying the assumptions of Proposition \ref{prop:entire} is to
prove Theorem \ref{thm:1} concerning the bounded  solution $u$
of \eqref{eq:1}, below we make no further reference to the solution
$u$ and just examine the entire solutions satisfying conditions
(ci)--(civ). Thus, Proposition \ref{prop:entire} can also be 
viewed as a statement concerning a class of entire solutions with some
additional properties, regardless of whether they belong to an
$\om$-limit set or not. Such a classification result for entire
solutions may be of independent interest. 

Let us first of all dispose of the trivial case $U_x\equiv 0$.
In this case, $U$, being
independent of  $x$, is a solution of $\dot\xi=f(\xi)$.
Condition (c0) implies that $U$ is nonconstant and it connects the
unstable equilibrium 0 to an equilibrium $\zeta$ with $(\zeta,0)\in
\La_{out}$. 
Thus in this case, \eqref{eq:6b} is proved.

In the remainder of this section, we assume
that $U$ is an entire solution of \eqref{eq:1} satisfying
conditions (c0)--(civ) and $U_x\not\equiv 0$.
We continue assuming the standing hypotheses (ND), (MF), and
\eqref{eq:17} on $f$, and simplify the  
notation letting  
$$\Lambda_{out}:=\Lambda_{out}(\Pi_0).$$

There are two possibilities in regard to the structure of $\La_{out}$
(cp. Figure \ref{pi0} in Section \ref{concepts}):
\begin{description}
\item[\bf(A1)] $\Lambda_{out}$ is a \emph{homoclinic loop,} that is,
  it is the union of a homoclinic orbit of \eqref{eq:sys} and its limit
  equilibrium, or, in other words,
  \begin{equation}
    \label{eq:21}
    \Lambda_{out}=\{(\ga,0)\}\  {\textstyle \bigcup }\  \tau(\Phi),
  \end{equation}
  where $f(\ga)=0$ and $\Phi$ is a ground state of \eqref{eq:steady} at
  level $\ga$. We choose $\Phi$ so that $\Phi'(0)=0$, that is, the
  only critical point of $\Phi$ is $x=0$.
\item[\bf(A2)] $\Lambda_{out}$ is a \emph{heteroclinic loop,} that is,
  it is the union of two heteroclinic orbits of \eqref{eq:sys} and their
  limit equilibria $(\gamma_\pm,0)$. In other words,
  \begin{equation}
    \label{eq:22}
    \Lambda_{out}=\{(\ga_-,0), (\ga_+,0)\}\ {\textstyle \bigcup}\ \tau(\Phi^+)\ {\textstyle \bigcup}\ \tau(\Phi^-),
  \end{equation}
  with $\gamma_-<\gamma_+$, $f(\ga_\pm)=0$, and $\Phi^\pm$ are
  standing waves of \eqref{eq:steady} connecting $\gamma_-$ and
  $\gamma_+$, one increasing the other one decreasing. We adopt the
  convention that $\Phi^+_x>0$ and $\Phi^-_x<0$.
\end{description}
To have a unified notation, we set
\begin{equation}
  \label{eq:19}
  \begin{aligned}
    \hat p&:=\inf\{a\in\R: (a,0)\in \Pi_0\}=\inf\{a\in\R: (a,0)\in \La_{out}\},\\
    \hat q&:=\sup\{a\in\R: (a,0)\in \Pi_0\}=\sup\{a\in\R: (a,0)\in
    \La_{out}\}.
  \end{aligned}
\end{equation}
Thus, $\{\hat p,\hat q\}=\{\ga, \Phi(0)\}$ if (A1) holds; and
$\hat p=\ga_-$, $\hat q =\ga_+$ if (A2) holds.

Note that if $\bar \ga$ is any of the constants
$\ga$, $\ga_\pm$ in (A1), (A2), then  $f'(\bar \ga)<0$.
Indeed, $f'(\bar\ga)=0$ is not allowed by  (ND), and $f'(\bar \ga)>0$
would imply that $(\bar \ga,0)$ is a center for \eqref{eq:sys}
and thus cannot be the limit equilibrium for any homoclinic or
heteroclinic orbit.

As for the limits \eqref{eq:29b},  since they
are solutions of the ordinary
differential equation $\dot\xi=f(\xi)$, if 
$\Theta(t)$ stands for $\Theta_-(t)$ or $\Theta_+(t)$,
then there are two possibilities:
 either $\Theta(t)=:\Theta$ is independent of $t$ and $f(\Theta)=0$,
 or else it is a
 strictly monotone solution. Due to \eqref{eq:5a},  in
 the former case  $(\Theta,0)$ equals $(0,0)$ or it is an element of
 $\La_{out}$, and in the latter case $\Theta(t)\ne 0$ for all $t$ and 
 $\Theta(-\infty)=0$, $(\Theta(\infty),0)\in\La_{out}$.
 We distinguish the following three cases:
 \begin{itemize}
\item[\bf(T1)] $\Theta_\pm(t)\ne 0$ for all $t$. 
\item[\bf(T2)] $\Theta_\pm\equiv 0$. 
\item[\bf(T3)]  $\Theta_+\equiv 0$ 
  and $\Theta_-(t)\ne 0$ for all
  $t$; or $\Theta_-\equiv 0$ and $\Theta_+(t)\ne 0$ for all
  $t$.
\end{itemize}

We treat these cases in separate subsections, proving \eqref{eq:6b} (and
sometimes more) in each of them.  The next subsection contains some
general lemmas that apply to all three cases. 

\subsection{Some general lemmas}
The following two lemmas show basic relations of
$U(\cdot,t)$ to  $\{(0,0)\}$ and $\Lambda_{out}$. They are special cases  (with
 $\Sigma_{in}=\{(0,0)\}$)  of \cite[Lemmas 4.7 and 
 4.8]{p-Pauthier2}.
Remember that we are assuming that
that $U$ satisfies 
conditions $U_x\not\equiv 0$ and (c0)--(civ)  (in particular $U$ is not
a steady state of \eqref{eq:1}). 
 
\begin{lemma}
  \label{le:basicrel}
  Let $K$ be any one of the sets $\{(0,0)\}$, $\Lambda_{out}$.  Then
  the following statements are 
  valid.
  \begin{itemize}
  \item[(i)] If $(x_n,t_n)$, $n=1,2,\dots,$ is a sequence in $\R^2$
    such that
    \begin{equation}
      \label{eq:58}
      \dist((U(x_n,t_n),U_x(x_n,t_n)),K)\to 0,
    \end{equation}
    then,
    possibly after passing to a subsequence, one has
    $U(\codt+x_n,\cdot+t_n)\to\varphi$ in $C^1_{loc}(\R^2)$, where
    $\varphi$ is a steady state of \eqref{eq:1} with
    $\tau(\varphi)\subset K$.
  \item[(ii)] There exists a sequence $(x_n,t_n)$, $n=1,2,\dots$ as in
    (i) with the additional property that $|t_n|\to\infty.$
    Consequently, there exists a steady state of \eqref{eq:1} with
    $\tau(\varphi)\subset K$ and
    \begin{equation}
      \label{eq:12}
      \varphi\in A(U)\cup \Om(U).
    \end{equation} 
  \end{itemize}
\end{lemma}
In the previous lemma, $\Om(U)$, $A(U)$
are the generalized limit sets of $U$, as defined in \eqref{defOmega},
\eqref{defAlpha}. Statement (i) of the lemma
delineates a way spatial trajectories
of $U$ can get arbitrarily
close to one of the sets $\{(0,0)\}$, $\Lambda_{out}$,
and statement (ii)  shows that the spatial 
trajectories of $U$ cannot stay away from from both
$\{(0,0)\}$ and $\Lambda_{out}$ as $|t|\to\infty$. In the next lemma,
we consider the case when the spatial trajectories stay away from
 one of the points  $(\hat p,0), (\hat q,0)\in \Lambda_{out}$
(see \eqref{eq:19} for the definition of $\hat p$, $\hat q$).

\begin{lemma} The following statements are valid.
  \label{le:ep}
  \begin{itemize}
  \item[(i)] If $U\le \hat q-\vartheta$ for some $\vartheta>0$,
    then $\om(U)=\{\hat p\}$ (so, necessarily,
    $f(\hat p)=0$) and $\al(U)=\{0\}$. Similarly, if
    $U\ge\hat p+\vartheta$ for some $\vartheta>0$, then $\om(U)=\{\hat
    q\}$ (so $f(\hat q)=0$) and 
    $\al(U)=\{0\}$.
  \item[(ii)] If for some $t_0\in \R$ and $\vartheta>0$ one has
    $U(\cdot,t)\le \hat q-\vartheta$ for all $t<t_0$, then
    $\al(U)=\{0\}$.  If for
    some $t_0\in \R$ and $\vartheta>0$ one has
    $U(\cdot,t)\ge \hat p+\vartheta$ for all $t<t_0$, then  $\al(U)=\{0\}$.
  \end{itemize}
\end{lemma}
We are making an intentional duplicity in this lemma
by including the conclusion $\al(U)=\{0\}$
in statement (i), although this conclusion is also contained  in
statement (ii) (which has
weaker assumptions). This will allow us to make more
straightforward references to one of these
statements. 

The next lemma gives other sufficient conditions for
$\al(U)=\{0\}$. It may be useful to reiterate at this point that
$\al(U)$ and
$\om(U)$ are considered with respect to the locally uniform
convergence. 

\begin{lemma}
  \label{le:alpha}
If one of the following conditions (a1)--(a4) is
satisfied, then $\al(U)=\{0\}$.
\begin{itemize}
\item[(a1)] $\Theta_+\not\equiv \hat p$ (so $\Theta_+>\hat p$) and 
  there is $m\in\R$ such that
  \begin{equation}
    \label{eq:8}
    \limsup_{x\in [m,\infty),\, t\to -\infty}U(x,t)\ (=
    \limsup_{t\to -\infty}\sup_{x\in [m,\infty),\, }U(x,t))\le 0.
  \end{equation}
  \item[(a2)] $\Theta_+\not\equiv \hat q$ and there is $m\in\R$ such
    that $\ \liminf_{x\in [m,\infty),\, t\to -\infty}U(x,t)\ge 0$.
  \item[(a3)] $\Theta_-\not\equiv \hat p$ and 
  and there is $m\in\R$ such that
   $\  \limsup_{x\in (-\infty,m),\, t\to -\infty}U(x,t)\le 0$.
  \item[(a4)] $\Theta_-\not\equiv \hat q$ and 
  and there is $m\in\R$ such that
   $\     \limsup_{x\in (-\infty,m),\, t\to -\infty}U(x,t)\ge 0$.
\end{itemize}
\end{lemma}
\begin{proof}
  We prove the conclusion assuming (a1) holds, all
  the other cases are analogous.
Note that \eqref{eq:8} in particular implies that 
  \begin{equation}
    \label{eq:7}
   \varphi(x)\le 0\quad (x\in (m,\infty),\ \varphi\in \al(U)). 
 \end{equation}
  
 To start with, we
 claim that no nonzero steady state of \eqref{eq:1} can belong to
$\al(U)$. To show this, it is sufficient---because of 
(c0)---to exclude from $\al(U)$ all steady states  $\varphi$ with
$\tau(\varphi)\subset \La_{out}$ (note that no nonconstant
periodic steady state can be contained in $\al(U)$ by (ci)).
Relation \eqref{eq:7} excludes the shifts of
the increasing standing wave $\Phi_+$ (if
$\La_{out}$ a heteroclinic loop as in (A2)), 
the constant $\hat q$ (if $f(\hat q)=0$), and the shift of the
ground state $\Phi$ in the case (A1) with $\ga=\Phi(\pm\infty)=\hat q$. 
It remains to exclude the constant
$\hat p$,  the shifts of the ground state $\Phi$
in the case (A1) with $\ga=\hat p$, and the shifts of the decreasing
standing wave $\Phi_-$ in the case (A2). The proofs for all these
use very similar comparisons arguments involving  periodic steady
states, so we give the details for just one of them, say for the
ground state $\Phi$ when  $\ga=\hat p$ (and $\Phi(0)=\hat q$).

Assume for a contradiction  that a shift $\tilde \Phi$ of
$\Phi$ is contained in $\al(U)$. 
Hence, there is sequence $t_n\to-\infty$
such that $U(\cdot,t_n)\to \tilde \Phi$ in $L^\infty_{loc}(\R)$. 
Let $x_0$ stand for the larger of the two zeros of $\tilde \Phi$. Clearly,
\eqref{eq:7} implies that $x_0\le m$ and $\tilde \Phi<0$ on $(m,\infty)$.
Pick any $\ep\in (0,\hat q)$. Then, by \eqref{eq:8}, there is
$t_0\in\R$ such that
\begin{equation}
  \label{eq:9}
  U(x,t)<\ep\quad (x\ge m,\ t\le t_0). 
\end{equation}
For any $\nu\in[\ep,\hat q),$ let
$\psi$ be any periodic solution of \eqref{eq:steady} with
$\tau(\psi)\subset\Pi_0$ and
$\psi(m)=\psi(m+\rho)=\nu$, $\rho>0$ being the minimal period of $\psi$. 
Then, $\psi(m)=\psi(m+\rho)>0 \ge \tilde \Phi(m)>\tilde \Phi(m+\rho)$;
and, in fact,
  $\tilde \Phi<\psi$ on
  $(m, m+\rho)$ (otherwise, a right shift of the graph of $\psi$ would be
  touching the graph of $\tilde \Phi$, which is impossible for two
  distinct solutions of \eqref{eq:steady}).  Consequently, if $n$ is
  large enough, we have $t_{n}<t_0-1$ and
  \begin{equation*}
    U(x,t_{n})<\psi(x)\quad(x\in(m,m+\rho)).
  \end{equation*}
  Since, by \eqref{eq:9}, we also have 
  \begin{equation*}
    U(m,t)<\psi(m)\ \textrm{ and } \ U(m+\rho,t)<\psi(m+\rho)\quad
 (t<t_0), 
  \end{equation*}
  applying the comparison principle on the domain
  $(m,m+\rho)\times(t_{n},t_0)$, we obtain
  \begin{equation*}
    U(x,t_0)<\psi(x)\quad (x\in(m,m+\rho)).
  \end{equation*}
  This is true for all periodic solutions $\psi$ with the indicated
  properties. We can choose a sequence of such periodic solutions
  converging locally uniformly to another shift $\bar \Phi:=\Phi(\cdot-m)$ of the
  ground state $\Phi$ (by continuity with respect to the initial data,
  the periods 
  $\rho$ of these periodic solutions  
  go to infinity). This implies that  $U(x,t_0)\le\bar
  \Phi(x)$ for all  $x>m$. In particular,
  $\Theta_+(t_0)=U(\infty,t_0)=\hat p$, in contradiction to the
  assumption on $\Theta_+(t_0)$. Our claim is proved.

  We now prove that $\al(U)=\{0\}$. Take any $\varphi\in \al(U)$.
  Let $\tilde U$ be the entire solution of \eqref{eq:1} with $\tilde
  U(\cdot,0)=\varphi$ and $\tilde U(\cdot,t)\in \al(U)$ for all
  $t\in\R$. By \eqref{eq:7}, $\tilde U(\cdot,t)\le 0$ in $(m,\infty)$
  for all $t\in\R$. 
  We go by contradiction: if $\varphi\not\equiv 0$, the strong comparison principle implies
  that 
  $\tilde U(\cdot,t)<0$ on $(m,\infty)$ for all $t>0$. Now,
  assumption \eqref{eq:17} and the Hamiltonian structure of system
  \eqref{eq:sys} imply that 
  $(0,0)$ is a center for \eqref{eq:sys}:  a
  neighborhood of $(0,0)$ is foliated by periodic orbits. We can thus
  choose a sequence $\psi_n$ of nonconstant periodic
  solutions of \eqref{eq:steady}, with their minimal periods
  bounded from above by a constant $\rho_0>0$,  such that $\max |\psi_n|\to 0$.  
  Pick any $x_0>m$ and denote
  \begin{equation*}
    s:=\max \{\tilde U(x,1):x\in [x_0,x_0+\rho_0]\}<0.
  \end{equation*}
  If $n$ is sufficiently large, then a suitable shift $\psi$ of
  $\psi_n$ satisfies 
  \begin{equation*}
   \psi(x_0)=\psi(x_1)=0\text{ and } s<\psi(x)<0 \quad (x\in (x_0,x_1)),
  \end{equation*}
  for some $x_1\in (x_0,x_0+\rho_0)$.
  Then $\tilde U(\cdot,1)\le \psi$ in $(x_0,x_1)$ and 
  $\tilde U(x_j,t)\le 0=\psi(x_j)$, $j=0,1$, for all $t\ge 1$. 
  Therefore, the comparison principle gives $\tilde U(\cdot,t)\le\psi$ on
  $(x_0,x_1)$ for all $t>1$. This and Theorem \ref{thm:GS} imply that
  $\om(\Tilde U)$, which is a subset of $\al(U)$ by compactness of
  $\al(U)$, contains a nonzero steady state, in contradiction to the above
  claim. We have thus shown that $\varphi\not\equiv 0$ is impossible,
  therefore   $\al(U)=\{0\}$.
\end{proof}

\subsection{Case (T1): $\Theta_\pm(t)\ne 0$ for all $t$
}\label{subsc1}
In this subsection, we assume that the entire solution $U$, fixed as
above, satisfies the relations $\Theta_\pm(t)\ne 0$ for all $t$.

We first show that $\tau(\Om(U))\subset \Lambda_{out}$
(note that this is stronger
than the needed conclusion  $\tau(\om(U))\subset \Lambda_{out}$). In
some cases, we even establish the existence of a limit
\begin{equation}
  \label{eq:3}
 \text{$\phi=\lim_{t\to\infty}U(\cdot,t)$ in $L^\infty(\R)$.}
\end{equation}
\begin{Lemma}
  \label{le:conv}
  The following statements are valid.
  \begin{itemize}
  \item[(i)] If 
  $\Theta_-$, $\Theta_+$  have opposite signs, then necessarily
    $\Lambda_{out}$ is a heteroclinic loop as in \emph{(A2)}, and the
    limit 
    $\phi$ in \eqref{eq:3} exists and is equal to a standing wave -- a
    shift of $\Phi^+$ or  $\Phi^-$.
  \item[(ii)] If the signs of
    $\Theta_-$, $\Theta_+$ are equal and $\Lambda_{out}$ is a
    heteroclinic loop as in \emph{(A2)}, then the
    limit 
    $\phi$ in \eqref{eq:3} exists and is equal to one of
    the constants $\ga_-$, $\ga_+$.
  \item[(iii)]  If the signs of
    $\Theta_-$, $\Theta_+$ are equal and 
  $\Lambda_{out}$ is a 
    homoclinic loop as in \emph{(A1)}, then the following statements
    hold:
    \begin{itemize}
    \item[(a)] If the functions
  $\Theta_-$, $\Theta_+$ are constant (so they are both identical 
  to the constant $\ga$), then the
    limit 
    $\phi$ in \eqref{eq:3} exists and is equal to the
    constant $\ga$ or a shift of the ground state $\Phi$.
  \item[(b)] If one (or both) of the functions functions
  $\Theta_-$, $\Theta_+$ is nonconstant, then $\tau(\Om(U))\subset \Lambda_{out}$.
    \end{itemize}
  \end{itemize}
  In particular, in all cases, we have $\tau(\om(U))\subset
  \Lambda_{out}$.
\end{Lemma}

\begin{remark}
  \label{rm:t1}{\rm
   The convergence conclusion as in
   statement (iii)(a) is likely valid in (iii)(b) as well. However,
   a result on threshold solutions from \cite{Matano_Polacik_CPDE16}
    that we are
    using in the proof of (iii)(a) does not seem to be available in
    general when $\Theta_-$, $\Theta_+$ are nonconstant.
    }
\end{remark}
\begin{proof}[Proof of Lemma \ref{le:conv}]
  Recall that the functions $\Theta_-$, $\Theta_+$ are solutions of
  the equation  $\dot\xi=f(\xi)$.  If they
  have opposite signs, as assumed in statement
  (i), then (whether they are
  constants or not) their limits
  \begin{equation}
    \label{eq:13}
    \theta_+=\lim_{t\to\infty}\Theta_+(t),\quad \theta_-=\lim_{t\to\infty}\Theta_-(t)
  \end{equation}
  are nonzero constants, both stable equilibria of
  $\dot\xi=f(\xi)$. Clearly, they have opposite signs, hence, in view of (c0), 
  necessarily 
  $\{\theta_-, \theta_+\}= \{\gamma_-, \gamma_+\}$, with $\ga_\pm$ as
  in (A2). Thus $U$ is a front-like solution in the sense that $U$ takes
  values between  $\ga_-$, $\ga_+$, and one of
  its spatial limits $U(\pm\infty,t)$ is in the interval $(\ga_-,0)$
  while the other one is in $(0,\ga_+)$.
  Since $f'(\ga_\pm)<0$, the convergence result in (i) is contained in
  \cite[Theorem 3.1]{FMcL}.

  Under the assumptions of (ii), the limits $\theta_-, \theta_+$ of 
  $\Theta_-(t)$, $\Theta_+(t)$ as $t\to\infty$ are equal to the same
  constant, either $\ga_-$ or $\ga_+$. In this case, 
  the convergence result stated in (ii) is also known and
  can easily be derived from
  \cite[Theorem 
  3.1]{FMcL} using a comparison function (see
   \cite[Proof of Lemma 3.4]{P:examples}, for example).

  Let now
  $\Lambda_{out}$ be a
  homoclinic loop as in (A1). For definiteness, we assume
  $\Phi(0)>\ga$ (the ground state $\Phi$ is above $\ga$). The case  
  $\Phi(0)<\ga$ is analogous. 
  Thus, in the notation introduced in \eqref{eq:19}, $\hat p=\ga$ and
  $\hat q=\Phi(0)$.

  Assume first that $\Theta_-\equiv \Theta_+\equiv \ga$. 
  Since $(\ga,\hat q]$ is the range of the ground state $\Phi$,
  $F<F(\ga)$ in $(\ga,\hat q]$ (here, $F(u)=\int_0^uf(s)\,ds$ is as
  in \eqref{energy}).   This is 
  the setup of \cite[Theorem 2.5]{Matano_Polacik_CPDE16} whose
  conclusion, translated to the present notation, is the same as the
  conclusion in (iii)(a).

It remains to prove statement (iii)(b). To that end, 
 we first claim 
  that  $(0,0)\not\in \tau(\Om(U))$. Observe that this claim implies
  the  desired inclusion $\tau(\Om(U))\subset
  \Lambda_{out}$. Indeed, 
  since 
  $\tau(\Om(U))$ is a compact subset of $\R^2$, our claim implies that
  there is a neighborhood of
  $(0,0)$ disjoint from $\tau(\Om(U))$. Consequently, since the
  equilibrium $(0,0)$ is
  a center for \eqref{eq:sys}, there is a nonconstant periodic orbit $\cO$ of
  \eqref{eq:sys} in this neighborhood  (with $(0,0)\in \cI(\cO)$). We
  then have
  $\tau(\Om(U))\subset\R^2\setminus{\overline{\mathcal{I}}(\mathcal{O})}$,
  and a direct application of Lemma \ref{le:squeezing} (taking
  \eqref{eq:5a} into account) gives $\tau(\Om(U))\subset
  \Lambda_{out}$.

  We prove our claim by 
  contradiction. Assume 
  $(0,0)\in \tau(\Om(U))$. This means that $0\in \Om(U)$, and hence
  there is a
  sequence $(x_n,t_n)$ with $t_n\to\infty$ such that
  $U(\cdot+x_n,t_n)\to 0$ in $L^\infty_{loc}(\R)$. Pick a nonconstant
  periodic solution $\psi$ of \eqref{eq:steady} with
  $\tau(\psi)\subset \Pi_0$, so that $(0,0)\in \cI(\tau(\psi))$ and
  $\hat p<\min \psi<\max\psi<\hat q$. 
  Then $\psi$ has infinitely many zeros and therefore
  $\psi-U(\cdot,t_n)$
  has infinitely many zeros for all large enough $n$. On the other
  hand,    each of the values
  $\Theta_\pm(t)=U(\pm\infty,t)$ is either equal to $\ga$ or, if nonconstant,
  approaches
  $\ga$  as $t\to\infty$. So both $\Theta_\pm(t)$ are 
  outside the interval $[\min\psi,\max\psi]$ for large $t$.
  Therefore, by Lemma
  \ref{le:zero},  $z (U(\cdot,t)-\psi)$ is finite and bounded for
  large $t$. We have obtained a contradiction, which proves our claim.   
\end{proof}

The following lemma completes the proof of Proposition
\ref{prop:entire} in the case (T1).
\begin{lemma}
  \label{le:t1alpha}
  $\al(U)=\{0\}$.
\end{lemma}

For the proof of this result, we need the following lemma.
\begin{lemma}
  \label{le:prep} The following statements are valid:
  \begin{itemize}
  \item[(j)] The zero number $z(U(\cdot,t))$ is
  independent of $t\in\R$. 
   \item[(jj)] If, for some $t\in\R$, $J$ is a nodal interval of
  $U(\cdot,t)$ 
  (that is, $U(\cdot,t)$ is nonzero everywhere  in $J$ and vanishes on
  $\partial J$), then  $U(\cdot,t)$ has at most one critical point in
  $J$ (the critical point is nondegenerate by (ci)). 
  \end{itemize}
\end{lemma}
\begin{proof}
  For (j), it is sufficient to prove that $t\mapsto z(U(\cdot,t))$ is locally
  constant.  
  Since the limits $\Theta_\pm(t)=U(\pm\infty,t)$ are nonzero, given any
  $T>0$ we can find $R>0$, $\ep>0$ such that $U(x,t)\ne 0$ for any
  $(x,t)\in\R^2\setminus [-R,R]\times [T-\ep,T+\ep]$. Since, by (ci),
  the zeros of $U(\cdot, t)$ are all simple, we obtain that $z(U(\cdot,t))$
  is independent of $t$, for $t\in [T-\ep,T+\ep]$.

  Statement (jj) follows from statement (j) and  conditions
  (ci), (civ).
\end{proof}

\begin{proof}[Proof of Lemma \ref{le:t1alpha}]
  If the functions $\Theta_\pm$ are both constant (so that statements
  (i),(ii), and (iii)(a) of Lemma \ref{le:conv} apply), the conclusion
  can be proved by nearly the same arguments as those given in
  the proof of 
  \cite[Lemma 4.10]{p-Pauthier2}. We omit the details in this case, just mention
  the following simple changes that need to made
  in the proof: take $\be_-=\be_+=0$, and replace
  all references to Lemmas 4.4 and 4.8 of \cite{p-Pauthier2} by references
  to Lemmas \ref{le:prep} and \ref{le:ep} of the present paper,
  respectively. 

  We now consider the case when at least one of the functions
  $\Theta_\pm$ is nonconstant. For definiteness, we assume that
  $\Theta_+$ is nonconstant. Here, too, there are two analogous
  possibilities:  $\Theta_+(t)\in (\hat p,0)$ for all $t$ and
  $\Theta_+(t)\in (0,\hat q)$ for all $t$. We just consider the
  former. 
  Note that, since $\Theta_+(t)$ converges to a stable equilibrium of
  $\dot\xi=f(\xi)$ as $t\to\infty$, we necessarily have $f(\hat
  p)=0$. 

  If $U<0$, we get the conclusion $\al(U)=\{0\}$ immediately
  from Lemma \ref{le:ep}(i). Henceforth assume that $U(\cdot,t)$
  has at least one zero for all $t$ and denote by
  $\eta(t)$ the largest of these
  zeros. Since we are assuming $\Theta_+(t)\in (\hat p,0)$,
  $U(\cdot,t)<0$ in $(\eta(t),\infty)$.
  Lemma \ref{le:prep}(j) and condition (ci) imply that $\eta(t)$ is a
  $C^1$ function of $t\in\R$. By Lemma \ref{le:prep}(jj),
  $U(\cdot,t)$ has at most one critical point in $(\eta(t),\infty)$.

  Due to the monotonicity of $t\mapsto
  z_{(\eta(t),\infty)}(U_x(\cdot,t))$---note that Lemma \ref{le:zerot}
  is applicable, as (ci) gives $U_x(\eta(t),t))\ne 0$---one of the following
  possibilities occurs:
  \begin{itemize}
  \item[(p1)] There is $t_0\in\R$ such $U(\cdot,t)$ has unique
    critical point in $ (\eta(t),\infty)$ for all $t<t_0$; this
    critical point  is the global minimizer of 
    $U(\cdot,t)$ in $ (\eta(t),\infty)$).
  \item[(p2)] $U_x(x,t)<0\quad (x>\eta(t))$.
  \end{itemize}
  
 If (p1) holds, then, by Lemma \ref{le:BPQ}, there is
 $m$ such $\eta(t)<m$ for all $t\in (-\infty,t_0]$.
 Since $U(\cdot,t)<0$ in  $(\eta(t),\infty)$, condition (a1) of
 Lemma \ref{le:alpha}
 applies and we obtain $\al(U)=\{0\}$.

  Assume now that (p2) holds.  If $U(\cdot,t)>\Theta_+(t)$ for all
  $t$, then, since $\Theta_+(t)\to 0$ as $t\to-\infty$, we obtain the
  desired conclusion  $\al(U)=\{0\}$ immediately from Lemma
  \ref{le:ep}(ii). Otherwise, $z(U(\cdot,t)-\Theta_+(t))\ge 1$ for
  some $t$, and consequently for all large negative $t$ (note that
  $U(\cdot,t)-\Theta_+(t)$ is a solution of a linear parabolic
  equation \eqref{eq:lin} on $\R^2$). The zero number is finite and
  bounded uniformly in $t$ due the bound on the number of critical
  points of $U(\cdot,t)$, see 
  (ci). Hence $z(U(\cdot,t)-\Theta_+(t))$
  is independent of $t$ for large negative $t$, say for $t<t_0$;
  and the zeros of $U(\cdot,t)-\Theta_+(t)$ are then all simple
  for $t<t_0$. We denote by $\xi(t)$ the largest of these zeros;
  $\xi(t)$ is a $C^1$ function of $t$. Clearly, for any $t<t_0$, we
  have $U(\cdot,t)-\Theta_+(t)>0$ on $(\xi(t),\infty)$ and 
  $U(\cdot,t)$ is not
  monotone in the interval $(\xi(t),\infty)$. Therefore, by Lemma \ref{le:BPQ},
  there is $m\in\R$ such that $\xi(t)\le m$ for all $t<t_0$.
  Using this and the fact that $\Theta_+(t)\to 0$ as $t\to-\infty$, we
  obtain 
  \begin{equation}
    \label{eq:8c}
    \liminf_{x\in [m,\infty),\, t\to -\infty}U(x,t)\ge 0.
  \end{equation}
  Applying Lemma \ref{le:alpha} , we
  conclude that $\al(U)=\{0\}$ in this case as well.
\end{proof}

\subsection{Case (T2): $\Theta_\pm\equiv 0$}

Throughout this subsection, we assume that
$\Theta_\pm\equiv 0$. Hence, $(U(x,t),U_x(x,t))\to 
  (0,0)$ as $x\to\pm\infty$ for all $t$.

If
$U>0$ or $U<0$, then the desired conclusion \eqref{eq:6b} follows from
Lemma \ref{le:ep}(i). Henceforth we therefore
assume that there is $t_0$ such
that 
\begin{equation}\label{eq:zg1}
    z( U(\cdot,t_0))\ge 1.
  \end{equation}
The following then holds. 

  \begin{lemma}
  \label{le:z1}
  Relation \eqref{eq:zg1} implies that for all $t<t_0$ one has
  \begin{equation}\label{eq:z1}
    z(U(\cdot,t))= 1.
  \end{equation}
\end{lemma}

\begin{proof}
  By monotonicity, the zero number in \eqref{eq:z1} is at least 1
  for all $t<t_0$. We prove that it is exactly 1, or, equivalently,
  that for any
  $t<t_0$ the spatial trajectory $\tau(U(\cdot,t))$ intersects
  the $v$-axis at
  exactly one point.
  We go by contradiction, assume that for some $t<t_0$ the function
  $U(\cdot,t)$ has at least two zeros and denote by $x_1<x_2$ the two
  smallest ones.  We also assume that $U(x,t)>0$ when $x<x_1$;
  the case $U(x,t)<0$  when $x<x_1$ being analogous.  Clearly,
  by property (civ) and the fact that $U(x,t)\to 0$  
  as $x\to-\infty$, $U(\cdot,t)$ has a unique critical point in each
  of the intervals $(-\infty,x_1)$, $(x_1,x_2)$; we denote them by
  $y_1$, $y_2$, respectively. We have $U_x(\cdot,t)>0$ on
  $(-\infty,y_1)\cup (y_2,x_2)$ and  $U_x(\cdot,t)<0$ on
  $ (y_1,y_2)$. 

\begin{figure}[h]
  \centering
  \hspace{-1.35cm} 
 \includegraphics[scale=.8]{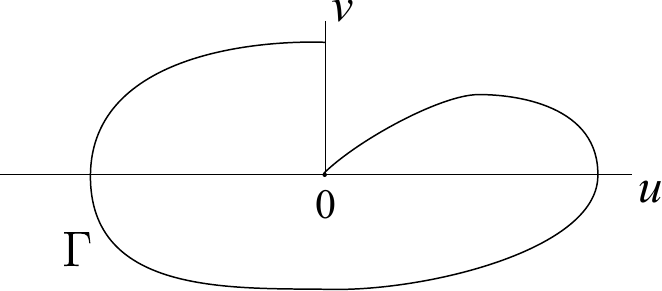}  
 \caption[The curve $\Ga$]{The curve $\Ga\subset\tau(U(\cdot,t))$
   whose existence is ruled out
    in the proof of Lemma \ref{le:z1}
  \label{contrad}}
\end{figure}

  Consider now the following curve, a part
  of the spatial trajectory $\tau(U(\cdot,t))$:
  \begin{equation*}
    \Ga:=\{(U(x,t),U_x(x,t)):x< x_2\}.
  \end{equation*}  
  By (ciii),  $\Ga$ is a simple curve. It is the union of the
  points $(U(y_1,t),0)$, $(U(y_2,t),0)$ on the $u$-axis, with
  $U(y_1,t)>0>U(y_2,t)$, 
  and the  sets
   $ \{(U(x,t),U_x(x,t)):x\le y_1\},\  \{(U(x,t),U_x(x,t)):y_1<x<
    y_2\}, \ \{(U(x,t),U_x(x,t)):y_2<x<
    x_2\}$, which are contained, respectively, in the quadrant $\{(u,v):u,v>
  0\}$, the half-plane $\{(u,v):v< 0\}$, and the quadrant 
  $\{(u,v):u<0<v\}$ (cp. Figure \ref{contrad}). 
  Since  $(U(x,t),U_x(x,t))\to 
  (0,0)$ as $x\to\pm\infty$, the union of $\Ga$ and the closed segment
  of the $v$-axis between the points $(0,U_x(x_2,t)$ and $(0,0)$, is a
  Jordan curve. Using now the facts that $\tau(U(\cdot,t)$ has no
  self-intersections (cp. (ciii)) and
  that 
  $U(x,t)$ increases with $x$ when $(U(x,t),U_x(x,t))$ is
  in the quadrant $\{(u,v):u,v> 0\}$, we obtain that
  $(U(x,t),U_x(x,t))$ cannot converge to $(0,0)$ as $x\to\infty$.
  We have thus found a contradiction, completing the proof.     
\end{proof}

Whenever \eqref{eq:z1} holds, conditions (civ), (ci), and (T2) imply
that the function $U(\cdot,t)$ has exactly two critical points,
both nondegenerate; one of them is the global
maximum point of $U(\cdot,t)$, further denoted by $\ol\xi(t)$,
the other one is the global minimum point of $U(\cdot,t)$, denoted by
$\ul\xi(t)$ (cp. Figure \ref{graphs-traj-U}). The unique zero of $U(\cdot,t)$ is denoted by $\xi(t)$;
it is between the two critical points.

By the monotonicity of the zero number,
\begin{equation}
  \label{eq:4}
  z(U(\cdot,t))\le 1\quad (t\in\R).
\end{equation}
 If $U(\cdot,t)>0$ for some $t$,
then $U(\cdot,t)$ has a unique critical point, the point of 
global maximum; we still denote it by $\ol\xi(t)$. If
$U(\cdot,t)<0$ for some $t$,
then the unique critical point of
$U(\cdot,t)$ is the point of 
global minimum, still denoted by $\ul\xi(t)$.

\begin{figure}[ht]
  \addtolength{\belowcaptionskip}{20pt}
  \centering
 \includegraphics[scale=.8]{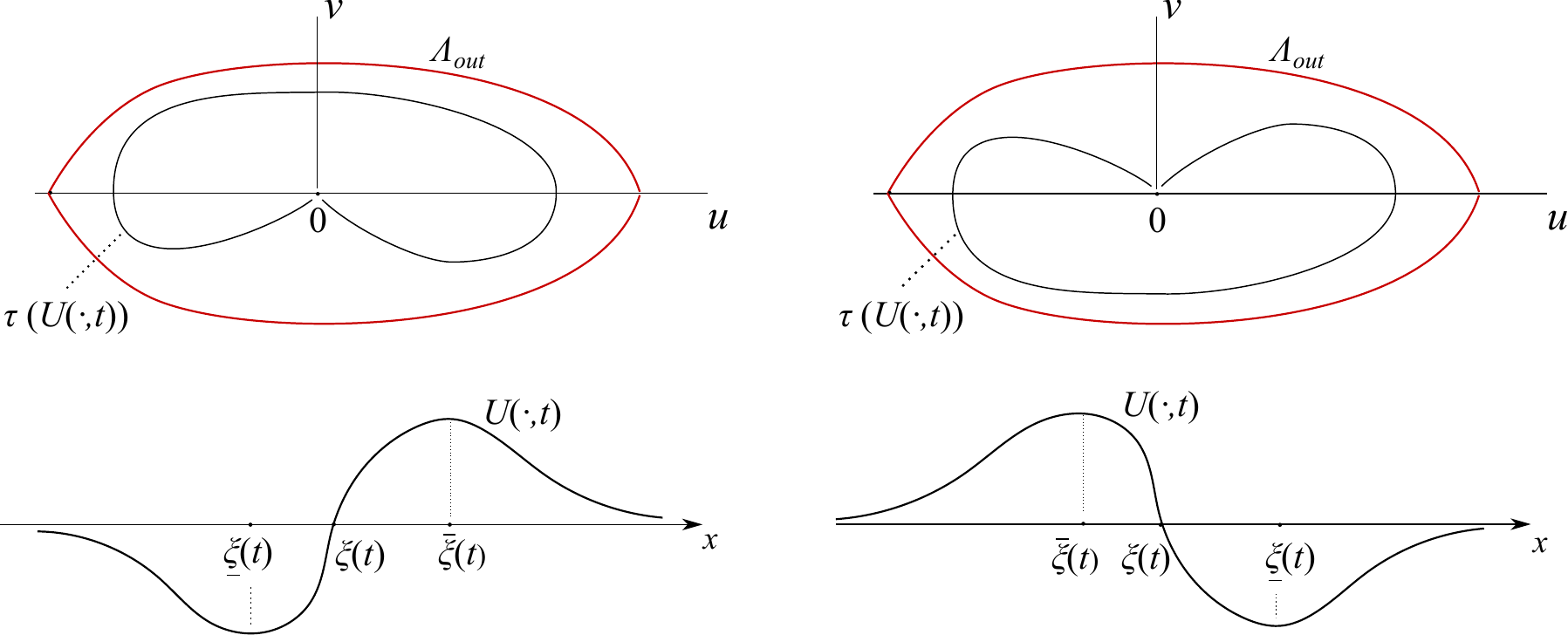}
  
  \caption[Graphs and spatial trajectories]{The spatial trajectories
    (top figures)
    and graphs of $U(\cdot,t)$ when $U(\cdot,t)$ has a zero $\xi(t)$
    (necessarily unique). The figures depict both possibilities
    $U(\cdot,t)<0$ in $(-\infty,\xi(t))$ (the figures on the left) and
    $U(\cdot,t)>0$ in $(-\infty,\xi(t))$. 
  \label{graphs-traj-U}}
\end{figure}

We now prove that the desired conclusion holds when
$\Lao$ is a homoclinic loop:
\begin{lemma}
~\\
  \label{le:finalc2a1}
  If $\Lao$ is a homoclinic loop as in \emph{(A1)}, then \eqref{eq:6b}
  holds: 
  $\al(U))=\{0\}$, $\tau(\om(U))\subset\Lao.$
\end{lemma}
\begin{proof}
  For definiteness, we assume that $\ga=\hp$, so $\Phi$ is a ground
  state at level $\hp$ and $\hq=\Phi(0)$ is its maximum; the case
  $\ga=\hq$ is similar. We prove that for some $\vartheta>0$
  \begin{equation}
    \label{eq:34}
    \sup_{x\in\R} U(x,t)<\hq-\vartheta\quad(t\in\R).
  \end{equation}
  Once this is done, the desired conclusion follows immediately from
  Lemma \ref{le:ep}(i).

  Assume that \eqref{eq:34} is not true for any
  $\vartheta>0$. Then there is a sequence $t_n\in\R$ such that
  $\sup_{x\in\R} U(\cdot,t_n)\upto \hq$. In particular,
  $U(\cdot,t_n)$ takes  positive values (for large enough $n$),
  so its global maximum point $\ol \xi(t_n)$ exists.
  We can therefore take $x_n=\ol\xi(t_n)$
  and thus have  $U_x(x_n,t_n)=0$. Using 
  Lemma \ref{le:basicrel}(i), we obtain that, passing to a subsequence if
  necessary, $U(\cdot+x_n,t_n)\to\Phi$ in 
  $C^1_{loc}(\R)$. But this implies that for large $n$ the function
  $U(\cdot+,t_n)$ has at least two zeros, a contradiction to \eqref{eq:4}.
  Thus \eqref{eq:34} holds and the proof is
  complete.
\end{proof}

Next, we treat the case when $\Lao$ is a heteroclinic loop.
\begin{lemma}
  \label{le:finalc2a2}
   Assume that $\Lao$ is a heteroclinic
  loop as in {\rm (A2)}.  Then \eqref{eq:6b}
  holds.
\end{lemma}

\begin{proof}
  We first show that 
  \begin{equation}
    \label{eq:35}
    A(U)=\{0\}.
  \end{equation}
  Note that this conclusion---stronger than the needed
  $\al(U)=\{0\}$---is equivalent to the convergence  $U(\cdot,t)\to
  0$ in $L^\infty(\R)$ (not just in $L_{loc}^\infty(\R)$) as $t\to-\infty$.

  With  $\ol\xi(t)$, $\ul\xi(t)$, $\xi(t)$ as above (cp. Figure
  \ref{graphs-traj-U}), and with both $\ol\xi(t)$ and $\ul\xi(t)$
  defined for $t<t_0$, we assume that
  \begin{equation}
    \label{eq:11}
    U(x,t)>0 \quad (x\in (-\infty,\xi(t)), \, t\in\R),\qquad
    U(x,t)<0 \quad (x\in (\xi(t),\infty), \, t\in\R).
  \end{equation}
   The case with the reversed inequalities  is analogous. 

  It is sufficient to prove that the constants $\ga_\pm$ are not
  contained in $A(U)$. Indeed, if this holds, then $A(U)$ does not
  contain any shifts of the standing waves $\Phi_\pm$ either (by
  compactness and translation invariance of $A(U)$).  Consequently, by
  Lemma \ref{le:basicrel},
  $\displaystyle \dist\left( \tau\left( A(U)\right),\Lambda_{out}
  \right)>0,$ and \eqref{eq:35} follows upon an application of Lemma
  \ref{le:squeezing}.

  Assume, for a contradiction that $\ga_+\in A(U)$ (arguments to rule
  out the possibility $\ga_-\in A(U)$ are similar and are omitted).
  Since the function
  $U(\cdot,t)$ is monotone neither on $(-\infty,\xi(t))$ nor on
  $(\xi(t),\infty)$, Lemma \ref{le:BPQ} tells us that for some
  $K>0$ one has
  \begin{equation}\label{eq:333}
    | \xi(t)|<K,\quad( t<t_0).
  \end{equation}
  From the assumption that $\ga^+\in A(U)$ we obtain that there are sequences
  $x_n$, $t_n$, with $t_n\to-\infty$, such
  that $U_n:=U(\cdot+x_n,\cdot+t_n)\to \ga^+$.  This, \eqref{eq:333},
  and \eqref{eq:11}  in  particular imply that 
  $x_n\to-\infty$.

  Let $\psi$ be any periodic solution of \eqref{eq:steady} with
  $\tau(\psi)\subset\Pi_0$ and $\psi(0)>0$, $\psi'(0)=0$.
  Let $2\rho>0$ be the minimal period of $\psi$, so
  $\psi(0)$ is the maximum of $\psi$, and
  $\psi(-\rho)=\psi(\rho)<0$ is the minimum of $\psi$.  Obviously,
  for all large enough $n$, say for all $n>n_0$, we have
  \begin{equation*}
    U(\cdot+x_n,t_n)>\psi\text{ on $[-\rho,\rho]$}.
  \end{equation*}
  Also, due to \eqref{eq:333} and the convergence $x_n\to-\infty$,  we
  have, making $n_0$ larger if necessary,
  \begin{equation*}
    U(\pm\rho +x_n,t)>0>\psi(-\rho)=\psi(\rho)\quad(n>n_0,\
    t\in (t_n,t_0]). 
  \end{equation*}
  Therefore, by the comparison principle, for $n>n_0$,
  \begin{equation*}
    U(x+x_n,t)>\psi(x)\quad (x\in [-\rho,\rho], \quad t\in (t_n,t_0]).
  \end{equation*}
  In particular, at $t=t_0$, we obtain
  \begin{equation*}
    \max_{x\in [-\rho,\rho]} U(x+x_n,t_0)\ge \max\psi>0 \quad
    (n>n_0).
  \end{equation*}
  Since $x_n\to -\infty$, we have a contradiction to the relation
  $U(-\infty,t_0)=\Theta_-=0$.  This contradiction completes the proof of
  \eqref{eq:35}.  

  We now prove the second needed conclusion:
  \begin{equation}
    \label{eq:10}
    \tau(\omega(U))\subset\Lao.
  \end{equation}
  If $U(\cdot,t)>0$ for some $t$, then $U(\cdot,t)\to\ga^+$ as
  $t\to\infty$ (with 
  the convergence in $L^\infty_{loc}(\R)$) by 
  the well known property of solutions with 
range in the monostable interval $(0,\ga^+)$. Similarly, if  
  $U(\cdot,t)<0$ for some $t$, then $U(\cdot,t)\to\ga^-$ as
  $t\to\infty$. In these cases we are done. We proceed assuming that
  $U(\cdot,t)$ changes sign for all $t$, that is, the equality holds
  in \eqref{eq:4} and $\ol\xi(t)$, $\ul\xi(t)$, $\xi(t)$ are defined
  for all $t\in\R$. For definiteness, we again assume that
  \eqref{eq:11} holds;
  the case with the reversed inequalities is analogous. 
  For this part of the proof, we 
  adapt some arguments
  from \cite[Proof of Lemma 4.13]{p-Pauthier2}.

  We claim that the following alternative holds:
  \begin{equation}\label{eq:341-3}
    \tau\left(\omega(U)\right)=\{(0,0)\} \textrm{ or } \tau\left(\omega(U)\right)\subset\Lambda_{out}.
  \end{equation} Indeed, relations \eqref{eq:11}
  imply that hypothesis (ii) of Theorem \ref{thmPP1} holds (namely,
  by \eqref{eq:11}, $V_\la U(x,t)>0$ for $x\approx \pm\infty$ and the
  finiteness of the zero number then follows from Lemma
  \ref{le:zero}(i)). From Theorem \ref{thmPP1} we obtain that  $\omega(U)$
  consists of steady states. It does not contain nonconstant
  periodic steady states (due to (ci)),  so
  $\tau(\omega(U))\subset \bar\Pi_0\setminus\mathcal{P}_0$. Since
  $\tau(\omega(U))$ is connected, \eqref{eq:341-3} must hold.

  Thus, to complete the proof of \eqref{eq:10}, we
  just need to rule out the possibility
  \begin{equation}
    \label{eq:40}
    \omega(U)=\{0\}. 
  \end{equation}
  Assume it holds. We derive a contradiction. Pick any  $\e>0$
  with $\ep<\min\{-\ga^-, \ga^+\}$. 
  Relation \eqref{eq:40} in particular implies that
  \begin{equation}\label{eq:341-6}
    \text{for any  $M>0$ there is $T=T(M)$ such that $
      -\e<U(x,t)<\e\quad(x\in(-M,M),\ t>T(M))$.}
  \end{equation}

  By Lemma \ref{le:basicrel}(ii), $\Om(U)\cup A(U)$
  contains one of the constants $\ga_\pm$ (or a shift of one of the
  standing waves $\Phi_\pm$, and, consequently, also both constants
  $\ga_\pm$). Since we have proved that $A(U)=\{0\}$,
  $\Om(U)$ must 
  contain one of these constants.
  We only consider the case $\gamma_+\in\Omega(U)$, the
  case $\gamma_-\in\Omega(U)$ being similar.  Hence, there is a
  sequence  $(x_n,t_n)$ with  $t_n\to\infty$ such that
  \begin{equation}\label{eq:341-5}
    U(\cdot+x_n,t_n)\underset{n\to\infty}{\longrightarrow}\gamma_+,
  \end{equation}
  with the convergence in $L^\infty_{loc}(\R).$ 
 In view of  \eqref{eq:341-6}, we have
  $|x_n|\to\infty$.  We claim that necessarily $x_n\to-\infty.$
  To show this, first observe that, by
  \eqref{eq:11} and  the simplicity of the zero $\xi(t)$, we have
  $U_x(\xi(t),t)<0.$ Using this  and
  Lemma \ref{le:BPQ}, we obtain the
  following monotonicity relations for each $t>0$:
  \begin{align}
    \textrm{if }\xi(t)>\xi(0), &\textrm{ then }U_x(\cdot,t)<0 \textrm{ on }\left(\xi(0),\xi(t)\right], \label{eq:341-42} \\
    \textrm{if }\xi(t)<\xi(0), &\textrm{ then }U_x(\cdot,t)<0 \textrm{ on }\left[\xi(t),\xi(0)\right). \label{eq:341-43} 
  \end{align}
  From \eqref{eq:11} and \eqref{eq:341-5}, it follows that there is
  $n_1$ such that 
  $\xi(t_n)>x_n$ for all  $n>n_1$. If for some $n>n_1$
   it is also true that $x_n>\xi(0)$, then the relations
  $\xi(t_n)>x_n>\xi(0)$ and 
  \eqref{eq:341-42} give $U(\xi(0),t_n)>U(x_n,t_n)$. This inequality
  can hold only for finitely many $n$, due to
  \eqref{eq:341-6}, \eqref{eq:341-5}. Thus for all large enough $n$ we
  have $x_n\le \xi(0)$. Since $|x_n|\to\infty$, it must be
  true that $x_n\to-\infty$, as claimed. 
  
  Pick now a periodic solution $\psi$ of \eqref{eq:steady} with
  $\tau(\psi)\subset\Pi_0$ such that $\min\psi<-\e$ and
  $\max\psi>\e$. We shift $\psi$ such that 
  $\max\psi=\psi(0)$. Let $2\rho>0$ be the
  minimal period of $\psi$; so 
  \begin{equation*}
    \min\psi=\psi(\pm\rho)<-\e,\qquad \max\psi=\psi(0)>\e.
  \end{equation*}
By \eqref{eq:341-5}, for $n$ large enough,
  \begin{equation}\label{dot}
    U(x_n+x,t_n)>\psi(0)>\e\quad(x\in(-\rho,\rho)).
  \end{equation}
This and \eqref{eq:11} in particular imply that $\xi(t_n)>x_n+\rho$. We 
now show that for some large enough $n_0$, the following must
  hold in addition to \eqref{dot}:
  \begin{equation}\label{eq:341-7}
    U(x_{n_0}\pm\rho,t)>\psi(\pm\rho)\quad( t>t_{n_0}).
  \end{equation}
  Indeed, if  not,  then there exists arbitrarily large $n$ such that 
  for some  $\tilde{t}_n>t_n$ one has
  $$\displaystyle U\left( x_n+\bar\rho,\tilde{t}_n\right)=\psi(\bar\rho)<-\e,$$
  where $\bar\rho$ is either $-\rho$ or $\rho$.
  Since $U(\cdot,t)>0$ on
  $(-\infty,\xi(t))$, it follows that
  $\displaystyle\xi\left(\tilde{t}_n\right)<x_n+\bar\rho$.
  But, due to
  $x_n\to-\infty$, we also have $x_n+\bar\rho<\xi(0)$ if $n$ is large
  enough; so, by \eqref{eq:341-43},
  $\displaystyle U\left(\xi(0),\tilde{t}_n\right)<U\left(
    x_n+\bar\rho,\tilde{t}_n\right)<-\e$. This cannot be true for
  arbitrarily large $n$, due to
  \eqref{eq:341-6}, 
  so \eqref{dot},\eqref{eq:341-7} both
  hold for some  $n_0$.

  Using \eqref{dot}, \eqref{eq:341-7}, and  the comparison principle,
  we obtain $U(x_{n_0},t)>\psi(0)>\e$
  for all $t>t_{n_0}.$  This is a contradiction to \eqref{eq:341-6}.

  We have shown that the assumption \eqref{eq:40} leads to a
  contradiction, which
  concludes the proof of Lemma \ref{le:finalc2a2}.
\end{proof}

\subsection{Case (T3): $\Theta_+\equiv 0$ 
  and $\Theta_-(t)\ne 0$ for all
  $t$.}
Dealing with case (T3), our assumption in this subsection 
is that $\Theta_+\equiv 0$ 
  and $\Theta_-(t)\ne 0$ for all
  $t$.
The other possibility, $\Theta_-\equiv 0$  and $\Theta_+(t)\ne 0$ for all  $t$ is analogous.  
For definiteness, we also assume that $\Theta_-<0$, hence
$\Theta_-(t)\to \hp$ 
as $t\to\infty$ (which includes the possibility that $\Theta_-\equiv \hp$)
and $f(\hp)=0$; the other possibility, $\Theta_->0$ can be treated
similarly.

First, we prove the desired conclusion regarding the $\al$-limit set
of $U$. 

\begin{lemma}\label{le:3.18}
 $\alpha(U)=\{0\}.$
\end{lemma}
\begin{proof} If
$U>0$ or $U<0$, then the conclusion follows from
Lemma \ref{le:ep}(i). Henceforth we therefore
assume that $z( U(\cdot,t))\ge 1$ holds for some $t$ and, consequently,
for any sufficiently large negative $t$. By (ci) and the monotonicity of the
zero number,  
$z(U(\cdot,t))$ is finite and independent of $t$
for large negative $t$  and all zeros of 
$U(\cdot,t)$ are simple. We denote by $\xi(t)$ the largest of these
zeros. Since $U(\infty,t)=\Theta_+(t)=0$, the function $U(\cdot,t)$ is not
  monotone in $(\xi(t),\infty)$. Therefore, by Lemma \ref{le:BPQ},
  $\xi(t)$ is bounded from above as $t\to-\infty$. 
  Thus, there are $t_0,m\in \R$ such that either $U(\cdot,t)<0$ on
    $(m,\infty)$ for all $t<t_0$ or $U(\cdot,t)>0$ on
    $(m,\infty)$ for all $t<t_0$. In either case, an application
    of Lemma \ref{le:alpha} gives $\al(U)=\{0\}$.
\end{proof}

We now consider the $\om$-limit set of $U$. The following
lemma,
in conjunction with Lemma \ref{le:3.18},
shows that \eqref{eq:6b} holds in the case
(T3). 

\begin{lemma}
  \label{le:c3om}
  $ \tau\left(\omega(U)\right)\subset\Lambda_{out}$.
\end{lemma}
\begin{proof}
  By  Theorem \ref{thmPP1} (which applies due to (T3)),
  $U$ is quasiconvergent, so
  $\omega(U)$ contains only nonperiodic steady states or constant
  steady states  (nonconstant periodic steady states are excluded by
  (ci)). By (c0), all $\varphi\in\om(U)$ satisfy $\tau(\varphi)\in
  \bar\Pi_0$, thus Lemma \ref{le:c3om} will be proved if we show that
  $0\not\in \om(U)$. To this end, we use a similar comparison
  arguments as in the proof of
  \cite[Lemma 4.14]{p-Pauthier2}.

  If $\Lao$ is a heteroclinic loop, as in (A2),  choose an increasing
  continuous function $\tilde u_0$ such that
  $0>\tilde u_0(-\infty)>\Theta_-(0)$,  $\tilde u_0(\infty) =\ga_+$, 
  and $\tilde u_0\ge U(\cdot,0)$. By the comparison principle,
  the corresponding solution
  $\tilde u=u(\cdot,\cdot,\tilde u_0)$ of \eqref{eq:1} satisfies
  $\tilde u(\cdot,t)>U(\cdot,t)$ for all $t>0$.  By
  \cite[Theorem 3.1]{FMcL}, the (front-like) solution
  $\tilde u(\codt,t)$ converges in $L^\infty(\R)$ to a shift of the
  increasing standing wave $\Phi^+$, say $\Phi^+(\codt-\eta)$, as
  $t\to\infty$.  This implies that $\varphi\le \Phi^+(\codt-\eta)$
  for all $\varphi\in\om(U)$; in particular $0\not\in \om(U)$.

  Assume now that $\Lao$ is a homoclinic loop, as in (A1).
  We have $\ga=\hp$ since, as noted above, $\Theta^-<0$ implies that
  $f(\hp)=0$. The ground state $\Phi$ satisfies $\Phi>\ga$,
  $\Phi(\pm\infty)=\ga$.  
  To prove that $0\not\in \om(U)$, we again use  a
  suitable
  comparison function $\tilde u_0$; specifically,  a bounded $C^1$ 
  function  $\tilde u_0$ with the following properties:
  \begin{itemize}
  \item[(s1)] $\tilde u_0$ has a unique critical point $x_0$ and
    $\tilde u_0(x_0)$ is the global maximum of $\tilde u_0$; 
    \item[(s2)] the limits  $\tilde u_0(-\infty)$, $\tilde
      u_0(\infty)$, satisfy the 
      following relations: 
      $$\Theta_-(0)<\tilde u_0(-\infty)<0,\quad \tilde
      u_0(\infty)=\ga;$$  
    \item[(s3)]  
    as $t\to\infty$,  the 
      solution   $\tilde u(\cdot,t):=u(\cdot,t,\tilde u_0)$ converges
      in $L^\infty(\R)$ to  $\Phi(\cdot-\mu)$,
      for some $\mu\in\R$. 
  \end{itemize}
  The existence of such a function $\tilde u_0$ is a 
  consequence of  \cite[Lemma
  3.5]{P:unbal}. Indeed, in \cite[Lemma
  3.5]{P:unbal}, the convergence to a shift $\Phi(\cdot-\mu)$  of the
  ground state $\Phi$  
  is proved for the solution $u(\cdot,t,g)$ whose initial datum $g$
  is given by the following relations (with  $\ga=\Phi(\pm\infty)$
  as above)
\begin{equation}
  \label{g-def}
  g(x)=\left\{ 
  \begin{alignedat}{2}
      \ &\beta\quad &&(x\in (-\infty,-q),\\
      \ &\vartheta\quad &&(x\in [-q,0]),\\      
\ &\ga \quad &&(x\in (0,\infty]),
    \end{alignedat}
\right.
\end{equation}
where $\be\in (\ga,0)$, $\vartheta, q\in (0,\infty)$ are constants;
$\be\in (\ga,0)$ can be chosen arbitrarily,  while
$\vartheta, q$ have to be selected suitably.
We fix any $\be$ with $\Theta_-(0)<\be<0$ and take the corresponding constants
$\vartheta, q>0$. Then for any sufficiently small $t_0>0$, the function
$\tilde u_0:=u(\cdot,t_0,g)$ is of class $C^1$ and satisfies (s1)
(cp. \cite[Remark 3.6]{P:unbal}).  For the limits at $x=\pm\infty$
we have the following relations:  $u(\infty,t,g)=\ga$ for each $t>0$ and
$u(-\infty,t,g)\to \beta$ as $t\downto 0$ (the function $t\mapsto
u(-\infty,t,g)$
is the solution of the equation  
$\dot \xi=f(\xi)$ with the initial condition $\xi(0)=\be$, see \cite[Lemma
3.3]{P:unbal} or \cite[Theorem 5.5.2]{Volpert3}).
Therefore,  (s2) holds for any sufficiently small $t_0>0$ as well. 
Finally, since
$u(\cdot,t,\tilde u_0)=u(\cdot,t_0+t,g)$ for $t>0$,
condition (s3) is satisfied.

Now, by (T3) and the assumptions made at the beginning of this
subsection, the function $U(x,0)$ satisfies the following relations:
$U(x,0)\to \Theta_-(0)$, $U_x(x,0)\to 0$ as $x\to-\infty$; 
$U(x,0)\to 0$ as $x\to\infty$. If $\Theta_-(0)=\ga$ (which holds if
and only if 
$\Theta_-(t)=\ga$ for
all $t$), then also $U_x(x,0)>0$ for all sufficiently large negative
$x$. These relations in conjunction with (s1) and (s2)
imply that  if $\eta>0$ is sufficiently
large, then the function $u_0(\cdot+\eta)-U(\codt,0)$
has only one zero and the zero is simple. 
It follows that 
$z(\tilde u(\cdot+\eta,t)-U(\codt,t))\le 1$ for all
$t>0$. Since $\tilde u(\cdot+\eta,t)\to \Phi(\cdot-\mu+\eta)$ as
$t\to\infty$, 
we obtain,  taking into account that the difference of
  any two steady states \eqref{eq:1} has only simple zeros, that
  $z(\Phi(\codt-\mu+\eta)-\varphi)\le 1$ for each $\varphi\in \om(U)$.
  In particular, $0\not\in \om(U)$.
\end{proof}

\bibliographystyle{amsplain}

\end{document}